\documentclass[reqno,12pt]{amsart}
\usepackage[latin1]{inputenc}
\usepackage[english]{babel}
\usepackage{amsmath, amsthm, amssymb, amsopn, amsfonts, amstext, enumerate, color, mathtools, mathrsfs, url, enumitem}

\usepackage{xcolor}

\usepackage[final]{showkeys}

\usepackage{hyperref}

\usepackage[margin=1.15in]{geometry}

\hypersetup{
	colorlinks=true,
	linkcolor=blue,
	citecolor=red,
	urlcolor=black,
	linktoc=all
}

\newcommand{\A}{\mathcal{A}}

\newcommand{\N}{\mathbb{N}}
\newcommand{\R}{\mathbb{R}}

\newcommand{\Z}{\mathbb{Z}}

\newcommand{\PV}{\mbox{\normalfont P.V.}}
\newcommand{\Haus}{\mathcal{H}}

\newcommand{\CC}{{\mathscr{C}}}

\newcommand{\delomega}{\partial\Omega}
\newcommand{\HalphaOmegat}{H_{\alpha}[\Omega_t]}

\DeclareMathOperator{\diam}{diam}
\DeclareMathOperator{\supp}{supp}

\DeclareMathOperator{\Per}{Per}

\def\XXint#1#2#3{{\setbox0=\hbox{$#1{#2#3}{\int}$ }
\vcenter{\hbox{$#2#3$ }}\kern-.6\wd0}}

\newlength{\dhatheight}

\numberwithin{equation}{section}

\theoremstyle{plain}
\newtheorem{definition}{Definition}[section]
\newtheorem{theorem}[definition]{Theorem}
\newtheorem{proposition}[definition]{Proposition}
\newtheorem{lemma}[definition]{Lemma}

\theoremstyle{definition}

\renewcommand{\le}{\leqslant}
\renewcommand{\leq}{\leqslant}
\renewcommand{\ge}{\geqslant}
\renewcommand{\geq}{\geqslant}

\title[A fractional Sobolev inequality on convex hypersurfaces]{A fractional Michael-Simon Sobolev inequality\\ on convex hypersurfaces}

\author{Xavier Cabr\'e}

\author{Matteo Cozzi}

\author{Gyula Csat\'o}

\address{\vspace{-\baselineskip}
\newline
\textit{Xavier Cabr\'e \textsuperscript{1,2}}
\newline
\textsuperscript{1}
ICREA, Pg. Lluis Companys 23, 08010 Barcelona, Spain
\newline
\textsuperscript{2} Universitat Polit\`ecnica de Catalunya, Departament de Matem\`{a}tiques, Diagonal 647, 08028 Barcelona, Spain
\newline
\textit{E-mail address}: \textit{\tt xavier.cabre@upc.edu}
}

\address{\vspace{-\baselineskip}
\newline
\textit{Matteo Cozzi \textsuperscript{3}}
\newline
\textsuperscript{3} University of Bath, Department of Mathematical Sciences, Claverton Down, Bath BA2 7AY, United Kingdom
\newline
\textit{E-mail address}: \textit{\tt m.cozzi@bath.ac.uk}
}

\address{\vspace{-\baselineskip}
\newline
\textit{Gyula Csat\'o \textsuperscript{4,5}}
\newline
\textsuperscript{4} Universitat de Barcelona,
Facultat de Matem\`atiques i Inform\`atica,
Gran Via de les Corts Catalanes 585,
08007 Barcelona, Spain
\newline
\textsuperscript{5} BGSMath Barcelona Graduate School of Mathematics,
Centre de Recerca Ma\-te\-m\`{a}\-ti\-ca,
Campus de Bellaterra, Edifici C,
08193 Bellaterra, Barcelona, Spain
\newline
\textit{E-mail address}: \textit{\tt gyula.csato@ub.edu}
}

\thanks{The first and the third authors are members of the Barcelona Graduate School of Mathematics, are part of the Catalan research 
group 2017 SGR 1392, and are supported by the MINECO grant MTM2017-84214-C2-1-P. 
The second author is supported by a Royal Society Newton International Fellowship.
The third author is also supported by the MINECO grant MTM2017-83499-P and by the ``Mar\'ia de Maeztu'' MINECO grant~MDM-2014-0445}



\begin{document}

\begin{abstract}
The classical Michael-Simon and Allard inequality is a Sobolev inequality for functions defined on a submanifold of Euclidean space. It is governed by a universal constant independent of the manifold, but displays on the right-hand side an additional $L^p$ term weighted by the mean curvature of the underlying manifold. We prove here a fractional version of this inequality on hypersurfaces of Euclidean space that are boundaries of convex sets. It involves the Gagliardo semi-norm of the function, as well as its $L^p$ norm weighted by the fractional mean curvature of the hypersurface. 

As an application, we establish a new upper bound for the maximal time of existence in the smooth fractional mean curvature flow of a convex set. The bound depends on the perimeter of the initial set instead of on its diameter.
\end{abstract}

\maketitle

\section{Introduction}
\label{section:intro}

\noindent
The Michael-Simon and Allard inequality is a Sobolev inequality on submanifolds of Euclidean space
which includes, on its right-hand side, an additional~$L^p$ integral weighted by a power of the submanifold's mean curvature norm. Remarkably, the presence of this extra geometric term enables the inequality to hold with a universal
constant independent of the manifold. As a consequence, this classical result has
important applications to the regularity of surfaces with prescribed mean curvature~\cite{BG73,Dierkes} and to the theory of geometric flows~\cite{Evans-Spruck}, among others.

In this article, we establish a fractional version of the inequality on convex hypersurfaces of Euclidean space---that is, hypersurfaces which are the boundary of an open convex set.
It involves the Gagliardo fractional semi-norm of a function defined on the surface, as well as an additional~$L^p$ norm weighted now by a power of the nonlocal mean curvature. As for its classical counterpart, our inequality carries a universal constant. The validity of a similar inequality in non-convex surfaces is still an open question.

Prior to this work, the only available fractional Michael-Simon and Allard inequality was established by the first two authors in~\cite{Cabre-Cozzi} for functions defined on nonlocal minimal surfaces. It was conceived and used in~\cite{Cabre-Cozzi} to derive a gradient estimate for nonlocal minimal graphs. Since nonlocal minimal
surfaces are never convex (except for hyperplanes), the result of~\cite{Cabre-Cozzi} and the one presented here complement each other.

As an application of the functional inequalities developed in the current paper, we obtain an upper bound on the maximal time of existence for the smooth fractional 
\linebreak 
$\alpha$-mean curvature flow of a convex set.
The fractional mean curvature flow was introduced by Caffarelli and Souganidis~\cite{Caffarelli Souganidis} and by Imbert~\cite{Imbert level set} in connection to diffusion phenomena with long range interactions. Similarly to the standard motion by mean curvature, bounded sets evolving according to this flow will become smaller after some time and ultimately disappear in finite time. A bound from above for the maximal time of existence of the smooth flow has been obtained in S\'aez and Valdinoci~\cite[Corollary~7]{Saez Valdinoci} by comparison with shrinking spheres. It reads
$$
T^\ast \le C \diam(\Omega_0)^{1 + \alpha},
$$
where~$\diam(\Omega_0)$ is the diameter of the initial set~$\Omega_0$ and the constant~$C$ depends only on~$n$ and~$\alpha$.

Assuming the initial set to be convex, Chambolle, Novaga, and Ruffini~\cite{Chambolle Novaga Ruffini} showed that convexity is preserved along the flow. By combining this fact with our fractional Michael-Simon type inequalities, we are able to improve, in the case of smooth convex evolutions, the aforementioned result of~\cite{Saez Valdinoci} to an estimate involving the area of the initial surface. Specifically, we prove that if~$\{ \Omega_t \}_{t \ge 0}$ is a family of~$C^{2}$ open subsets of~$\R^{n + 1}$ evolving by fractional~$\alpha$-mean curvature flow, with~$\Omega_0$ convex, then the maximal time of existence~$T^\ast$ satisfies
$$
T^\ast \le C \hspace{1pt} |\partial \Omega_0|^{\frac{1 + \alpha}{n}},
$$
for some constant~$C$ depending only on~$n$ and~$\alpha$.

\subsection{The classical Michael-Simon and Allard inequality}
This inequality is an extension of the classical Sobolev inequality to~$m$-dimensional submanifolds of~$\R^{n+1}$. It  was proved in the seventies independently by Allard~\cite{Allard} and by Michael and Simon~\cite{Michal Simon}---the latter for a class of generalized submanifolds, the former in an even broader varifold setting. The  following is the statement in the context of~$C^2$ hypersurfaces~$M \subset \R^{n + 1}$. It makes no assumption on the topology of~$M$, in particular whether it is compact or not. We denote the space of~$C^1$ functions in~$M$ with compact support by~$C_c^1(M)$, which agrees with~$C^1(M)$ when~$M$ is compact. 

\begin{theorem}[Allard \cite{Allard}, Michael and Simon \cite{Michal Simon}]
\label{thm:Michael-Simon Allard}
Let~$n\geq 2$ be an integer,~$p\in [1,n)$, and~$M\subset\R^{n+1}$ a~$C^2$ hypersurface.
Then, there exists a constant~$C$ depending only on~$n$ and~$p$, such that
\begin{equation}
 \label{eq:intro:classical MS}
  \|u\|_{L^{p^{\ast}}(M)}\leq C \, \Big( \|\nabla_{\! M} u\|_{L^p(M)}+\|Hu\|_{L^p(M)} \Big) \quad
  \text{for all }u\in C_c^1(M),
\end{equation}
where~$p^{\ast} := {np}/({n-p}),$~$\nabla_{\! M}$ is the tangential gradient on~$M$, and~$H$ is the mean curvature of~$M.$
\end{theorem}

We refer the reader to the recent paper~\cite{Cabre-Miraglio} of Miraglio and the first author where, combining the ideas of~\cite{Allard,Michal Simon}, a quick and easy-to-read proof of Theorem~\ref{thm:Michael-Simon Allard} is provided.

Exactly as for the Euclidean Sobolev inequality, Theorem~\ref{thm:Michael-Simon Allard} can be deduced, using the coarea formula and H\"older's inequality, from the case when~$p=1$ and~$u=\chi_{E}$ is the characteristic function of a sufficiently regular subset~$E\subset M.$ For these choices, inequality~\eqref{eq:intro:classical MS} is an isoperimetric one and reads as
\begin{equation}
 \label{eq:intro:MS geometric version}
  |E|^{\frac{n-1}{n}}\leq C \left( \Per_M(E)+\int_E |H(x)| \, dx \right) ,
\end{equation}
where~$|E|$ stands for the~$n$-dimensional Hausdorff measure of~$E$,~$dx$ indicates the restriction of such measure to~$M$, and~$\Per_M(E)$ denotes the perimeter of~$E$ in~$M$.

We emphasize that the constant~$C$ does not depend on~$M$ and that therefore all the information about the geometry of~$M$ is captured by its mean curvature~$H$ appearing on the right-hand side of~\eqref{eq:intro:classical MS}. In particular, if~$M$ is a minimal surface, i.e., if~$H=0$, then estimate~\eqref{eq:intro:classical MS} holds true with only~$\|\nabla_{\! M} u\|_{L^p(M)}$ appearing on its right-hand side, exactly as in the Euclidean case. Such a universal Sobolev inequality on minimal surfaces was first obtained by Bombieri, De Giorgi, and Miranda~\cite{BDM}---and consequently prior to that of Michael-Simon and Allard.

To determine the best constant in~\eqref{eq:intro:MS geometric version} remained as an open question for many years, even when~$H\equiv 0.$ In a very recent paper, Brendle~\cite{Brendle} has proved that, in every minimal surface,~\eqref{eq:intro:MS geometric version} holds true taking~$C$ to be the isoperimetric constant in~$\R^n$. Moreover, equality is achieved only by flat~$n$-dimensional balls. Brendle's argument is a far-reaching extension of the proof of the Euclidean isoperimetric inequality via the Aleksandrov-Bakelman-Pucci method found by the first author---see,~e.g.,~\cite{Cabre-DCDS}.

Another interesting class of hypersurfaces are those which are compact (with no boundary). In this case, one can plug~$u\equiv 1$ into~\eqref{eq:intro:classical MS}. This leads to an estimate from below for the integral of the modulus of the mean curvature of~$M$ in terms of the measure of~$M$:
\begin{equation}
 \label{eq:intro:Alexandrov Fenchel nonconvex}
  |M|^{\frac{n-1}{n}}\leq C \int_M |H(x)| \, dx.
\end{equation}
For a convex hypersurface~$M$---that is, when~$M=\delomega$ is the boundary of a convex subset~$\Omega$ of~$\R^{n+1}$---estimate~\eqref{eq:intro:Alexandrov Fenchel nonconvex} is a particular case of the Aleksandrov-Fenchel inequalities. In this convex case, it is known to hold with the optimal constant---which is achieved by all spheres~$M=\partial B_R(x)$.  However, it is still an open problem to determine the optimal constant for general compact hypersurfaces. See~\cite{Aleksandrov 1,Aleksandrov 2} and also Chang and Wang~\cite{Alice Chang} for a recent survey on this topic.

\subsection{A fractional Michael-Simon and Allard inequality on convex hypersurfaces}

It is a well-known fact that an appropriate Sobolev embedding holds for Sobolev spaces of fractional order in Euclidean space. Indeed, for every~$s \in (0,1)$, every integer~$n\geq 1$, and every~$p\in [1,n/s)$, there exists a constant~$C$ depending only on~$n$,~$p$, and~$s$, such that
\begin{equation}\label{eq:intro:frac Sobolev euclidean}
  \|u\|_{L^{p^{\ast}_s}(\R^n)}\leq C \,[u]_{W^{s,p}(\R^n)}\quad\text{for all }u\in W^{s, p}(\R^n).
\end{equation}
Here~$W^{s, p}(\R^n)$ is the fractional Sobolev space of functions~$u \in L^p(\R^n)$ for which the Gagliardo semi-norm
\begin{equation} \label{Wspsemidef}
  \quad [u]_{W^{s,p}(\R^n)} := \left(\int_{\R^n}\int_{\R^n}\frac{|u(x)-u(y)|^p}{|x-y|^{n+s p}} \, dx dy\right)^{\frac{1}{p}}
\end{equation}
is finite and
$$
p^{\ast}_s := \frac{np}{n- s p}
$$
is the relevant critical Sobolev exponent.

Historically, fractional Sobolev spaces were introduced to measure the smoothness of functions defined on~\emph{curved} hypersurfaces of Euclidean spaces, with special interest on boundaries of bounded open Lipschitz sets~$\Omega \subset \R^{n + 1}$. Indeed, Aronszajn~\cite{A55}, Slobodecki{\u{\i}} and Babi{\u{c}}~\cite{SB56}, and Gagliardo~\cite{G57} showed that, for~$p > 1$, the trace space of~$W^{1, p}(\Omega)$ is~$W^{(p - 1)/p, p}(\partial \Omega)$. The fractional Sobolev space~$W^{s, p}(M)$ on a hypersurface~$M\subset\R^{n+1}$ can be defined, similarly to the Euclidean case, as the collection of~$L^p(M)$ functions having finite semi-norm~$[\, \cdot \,]_{W^{s, p}(M)}.$ This semi-norm is defined as in~\eqref{Wspsemidef} by replacing the domain of integration~$\R^n$ with~$M$, writing~$dx$ to mean integration with respect to the~$n$-dimensional Hausdorff measure, and understanding~$|x-y|$ to be the standard Euclidean distance in~$\R^{n+1}$:
$$
  [u]_{W^{s,p}(M)}:=\left(\int_M\int_M\frac{|u(x)-u(y)|^p}{|x-y|^{n+sp}} \, dxdy\right)^{\frac{1}{p}}.
$$

In the current paper, we study the existence of a version of the Michael-Simon and Allard inequality for fractional Sobolev spaces on hypersurfaces of Euclidean space. Our interest originates from the theory of nonlocal minimal surfaces. Given~$\alpha \in (0, 1)$, nonlocal~$\alpha$-minimal surfaces are defined as being (the boundaries of) the critical points of the fractional~$\alpha$-perimeter functional
\begin{equation}\label{def:alpha perimeter flat}
\R^{n + 1} \supset \Omega \longmapsto \Per_\alpha(\Omega) := \frac{1}{2} \, [\chi_\Omega]_{W^{\alpha, 1}(\R^{n + 1})} = \int_\Omega \int_{\R^{n + 1} \setminus \Omega} \frac{dx dy}{|x - y|^{n + 1 + \alpha}}.
\end{equation}
They were introduced by Caffarelli, Roquejoffre, and Savin in~\cite{Caffarelli Roquejoffre Savin}, and are related to phase-transition models with strongly nonlocal interactions. Such critical points are characterized by the equation
$$
H_\alpha[\Omega] = 0 \quad \mbox{on } \partial \Omega,
$$
where
\begin{align} \label{Halphadef}
H_\alpha(x) =  H_\alpha[\Omega](x) := \hspace{3pt} & \frac{\alpha}{2} \, \PV \int_{\R^{n+1}}\frac{\chi_{\R^{n + 1} \setminus \Omega}(y)-\chi_{\Omega}(y)}{|y-x|^{n+1+\alpha}} \, dy
\\
= \hspace{3pt} &
 \label{eq:intro:def of H alpha}
   \PV \int_{\delomega} \frac{(y - x) \cdot \nu(y)}{|y - x|^{n + 1 + \alpha}} \, dy\qquad\text{for }x\in\delomega
\end{align}
is the so-called nonlocal (or fractional)~$\alpha$-mean curvature of~$\Omega$ at the point~$x \in \partial \Omega$ and~$\nu$ denotes the exterior unit normal vector to~$\partial\Omega$. Note that the last equality follows from the divergence theorem.
It is known that these surfaces satisfy a density estimate
\begin{equation}
 \label{eq:intro:density in Cabre Cozzi thm}
  |M\cap B_R(x_0)|\geq c_{\ast} R^n\quad\text{for all~$x_0\in M := \partial \Omega$ and~$R>0$},
\end{equation}
as in the case of standard minimal surfaces. Here, the positive constant~$c_{\ast}$ depends only on~$n$ and~$\alpha$.

It was the study of nonlocal $\alpha$-minimal surfaces that led to the first result on a fractional Michael-Simon inequality, obtained by the first two authors. In~\cite{Cabre-Cozzi}, we obtained the following new universal fractional Sobolev inequality on nonlocal~$\alpha$-minimal surfaces, as well as on classical minimal surfaces. We established it by extending a beautiful proof of the fractional Sobolev inequality in Euclidean space due to Brezis~\cite{brezis}. In~\cite{Cabre-Cozzi} such result played a central role in the proof of a gradient estimate for nonlocal minimal graphs.

\begin{theorem}[Cabr\'e and Cozzi \cite{Cabre-Cozzi}]
\label{thm:intro:Cabre Cozzi fract Sobolev}
Let~$n\geq 1$ be an integer,~$s \in (0,1)$, and~$p\geq 1$ be such that~$n> s p$. Let~$M\subset\R^{n+1}$ be either a nonlocal~$\alpha$-minimal surface or a classical minimal surface---more generally, it suffices to assume that~$M\subset\R^{n+1}$ is a set with locally finite~$n$-dimensional Hausdorff measure that satisfies~\eqref{eq:intro:density in Cabre Cozzi thm} for some positive constant~$c_{\ast}$.

Then, there exists a constant~$C$ depending only on~$n$,~$s$,~$p$, and~$c_{\ast}$, such that
\begin{equation}
 \label{eq:intro:Cabre Cozzi frac Sob ineq}
  \|u\|_{L^{p^{\ast}_s}(M)}\leq C \, [u]_{W^{s,p}(M)}\quad\text{for all } u\in W^{s, p}(M).
\end{equation}
\end{theorem}

Also recently, and independently from~\cite{Cabre-Cozzi}, inequality~\eqref{eq:intro:Cabre Cozzi frac Sob ineq} has been obtained by Dyda et al.~\cite{DILTV18} as part of a more general family of Hardy-Sobolev type inequalities for weighted fractional Sobolev spaces defined on metric measure spaces---see~\cite[Theorem~5.3]{DILTV18}. However, when restricted to a hypersurface~$M$ of Euclidean space, their inequalities hold under stronger assumptions than the density estimate \eqref{eq:intro:density in Cabre Cozzi thm}---namely, a connectivity type hypothesis on~$M$ and the validity of quantitative doubling and reverse doubling conditions on the~$n$-dimensional Hausdorff measure restricted to~$M$, in addition to~\eqref{eq:intro:density in Cabre Cozzi thm}.

In light of Theorem~\ref{thm:intro:Cabre Cozzi fract Sobolev} and the classical Michael-Simon inequality, it is conceivable that~\eqref{eq:intro:Cabre Cozzi frac Sob ineq} could be extended to general hypersurfaces by including an additional remainder~$L^p$-term involving the nonlocal mean curvature. The following result---which is the main contribution of our paper---shows that this is indeed the case for convex hypersurfaces. The question remains open in the non-convex case.

In order to state our theorem, note first that if~$\Omega\subset\R^{n+1}$ is an open convex set, then~$\partial\Omega$ is a Lipschitz hypersurface and thus, by Rademacher's theorem, differentiable at almost every point~$x\in\partial\Omega.$ On the other hand, by either expression~\eqref{Halphadef} or~\eqref{eq:intro:def of H alpha}, we see that~$H_{\alpha}(x)$ is a well-defined quantity in~$[0,+\infty],$ since~$\Omega$ is convex. Furthermore, by Aleksandrov's theorem,~$\partial\Omega$ is (pointwise) twice differentiable at almost every~$x\in\partial\Omega.$ At these points, the nonlocal mean curvature~$H_{\alpha}(x)$ is finite.

Note also that every bounded open convex set~$\Omega\subset\R^{n+1}$ has finite perimeter, that is,~$|\partial \Omega|<+\infty$. This follows from the classical isodiametric inequality for the perimeter of convex sets, stated in Proposition~\ref{RSineprop} and proved in Appendix~\ref{appendix:rosenthal-szasz}. This fact will be important within some of our proofs, to avoid the indetermination~$0\cdot\infty$.

We also need to define the nonlocal~$\alpha$-perimeter functional on hypersurfaces, as the natural generalization of the Euclidean nonlocal perimeter functional~\eqref{def:alpha perimeter flat}. Given a hypersurface~$M \subset \R^{n + 1}$ and a subset~$F \subset M$, for~$s \in (0, 1)$ we define
$$
\Per_{M, s}(F) := \int_{F} \int_{M \setminus F} \frac{dx dy}{|x - y|^{n + s}}.
$$

\begin{theorem}
\label{thm:intro:Convex MS for functions} Let~$n \ge 1$ be an integer, $\alpha\in(0,1),$ $s\in(0,1),$  and~$p\geq 1$ be such that~$n > s p$. Let~$\Omega\subset\R^{n+1}$ be an open convex set. Then, there exists a constant~$C$ depending only on~$n$,~$\alpha$,~$s$, and~$p$, such that
\begin{equation} \label{eq:intro:sAMSonconvex}
  \|u\|_{L^{p^{\ast}_s}(\delomega)}\leq C\left(\frac{1}{2}\int_{\delomega}\int_{\delomega} \frac{|u(x)-u(y)|^p}{|x-y|^{n+ s p}} \, dx dy +\int_{\delomega}H_{\alpha}(x)^{\frac{s p}{\alpha}} |u(x)|^p \, dx
  \right)^{\frac{1}{p}}
\end{equation}
holds true for every~$u \in W^{s, p}(\partial \Omega)$, where~$p_s^{\ast}=np/(n-sp)$.

As a consequence, taking~$p=1$ and $u$ to be the characteristic function of a set~$E,$
we have
\begin{equation}\label{old theorem frac MS for sets}
  |E|^{\frac{n - s}{n}} \le C \left( \Per_{\delomega, s}(E) + \int_E H_{\alpha}(x)^{\frac{s}{\alpha}} \, {dx} \right)
\end{equation}
for every measurable subset~$E \subset \delomega$ with finite measure, where~$C$ is a constant depending only on~$n$,~$\alpha$, and~$s$. In particular, if~$\partial \Omega$ has finite measure, the choice~$E=\delomega$ leads to
\begin{equation}
 \label{old thm frac Aleks-Fench}
  |\partial\Omega|^{\frac{n-s}{n}}\leq C\int_{\partial\Omega}H_{\alpha}(x)^{\frac{s}{\alpha}} \, {dx}.
\end{equation}
\end{theorem}

Note that no relation between the parameters~$s$ and~$\alpha$ is assumed within the theorem.

Inequality  \eqref{old thm frac Aleks-Fench} is a fractional extension of the classical Aleksandrov-Fenchel type inequality~\eqref{eq:intro:Alexandrov Fenchel nonconvex}. In \eqref{old thm frac Aleks-Fench}, it is still unknown what the best constant is. In addition, its validity in non-convex surfaces---after replacing $H_{\alpha}$ by $|H_{\alpha}|$---remains as an open question.

Observe that neither Theorem~\ref{thm:intro:Cabre Cozzi fract Sobolev} is a particular case of Theorem~\ref{thm:intro:Convex MS for functions} (since the former makes no convexity assumption, and thus includes classical and nonlocal minimal surfaces), nor one can deduce the latter from the former. 
Indeed, Theorem~\ref{thm:intro:Convex MS for functions} holds not only for unbounded convex sets but also for bounded ones, and in this last case, the density estimate~\eqref{eq:intro:density in Cabre Cozzi thm} cannot hold because~$|\delomega\cap B_R(x_0)|=|\delomega|$ for all sufficiently large balls~$B_R(x_0)$. Note also that, by convexity, the fractional mean curvature~$H_\alpha$ is positive at all points of~$\partial \Omega$ except if~$\partial \Omega$ is a hyperplane.

\subsection{An application to the fractional mean curvature flow of convex sets}

In Section~\ref{section:fractional mean curv flow} we give an application of our results to get a new bound on the maximal time of existence for the smooth fractional mean curvature flow of convex hypersurfaces. For this, we will use the pointwise inequality~\eqref{eq:intro:pointwise ineq delomega} reported below---which is a key tool within the proof of our main result, Theorem 1.3--- as well as the classical Michael-Simon inequality.

Without entering into regularity issues, a family of open sets~$\{ \Omega_t \}_{t \ge 0}$ evolves by fractional~$\alpha$-mean curvature if the inner normal velocity at a point~$x \in \delomega_t$ is equal to the fractional~$\alpha$-mean curvature of~$\Omega_t$ at~$x$. This flow has been investigated recently in several works. The existence and uniqueness of viscosity solutions to the generalized level set flow was obtained by Imbert~\cite{Imbert level set}. Julin and La Manna~\cite{JulinLaManna} established that, if the initial set~$\Omega_0$ is bounded and of class~$C^2$, then~$\Omega_t$ is smooth for sufficiently small times~$t$. Thus, the time
\begin{equation}\label{max-time}
T^\ast := \sup \left\{ t > 0 : \Omega_\tau \mbox{ is non-empty and has~$C^2$ boundary for all~} \tau\in [0,t) \right\}
\end{equation}
is positive. The sets~$\Omega_t$ will become empty in finite time, possibly developing singularities prior to extinction. Indeed, S\'aez and Valdinoci~\cite[Corollary 7]{Saez Valdinoci} have shown that
\begin{equation} \label{SaezVal_bound}
T^\ast \le C \diam(\Omega_0)^{1 + \alpha},
\end{equation}
for some constant~$C$ depending only on~$n$ and~$\alpha$, whereas an example in which singularities arise before extinction for a non-convex initial datum has been produced by Cinti, Sinestari, and Valdinoci~\cite{CintiSinestrariValdinoci}.\footnote{We stress that~$T^\ast$ is the maximal time for which the fractional~$\alpha$-mean curvature flow originating from a $C^{2}$ convex set~$\Omega_0$ remains $C^{2}$. Due to the possible formation of singularities, this time might be in principle smaller than the extinction time~$T_e := \sup \left\{ t > 0 : \Omega_t \mbox{ is non-empty} \right\}$ of the generalized level set flow considered, e.g., in~\cite{Imbert level set,Chambolle Novaga Ruffini}. Since, for $\Omega_{0}$ convex, this possibility has not been ruled out at the current time (in contrast with the case of convex sets evolving by classical mean curvature flow~\cite{huisken,Evans-Spruck}), we keep this distinction. We also point out that,  via results and ideas from~\cite{CDNV19,Imbert level set},
the upper bound on~$T^\ast$ provided by~\cite{Saez Valdinoci} can be actually improved to an estimate on the extinction time~$T_e$, regardless of the convexity of the initial datum.
\label{TastTefoot}}

When~$\Omega_0$ is convex, then each~$\Omega_t$ is convex as well, as shown by Chambolle, Novaga, and Ruffini~\cite{Chambolle Novaga Ruffini}. Thanks to this observation, by combining the pointwise nonlocal estimate~\eqref{eq:intro:pointwise ineq delomega} with the classical Michael-Simon inequality of Theorem~\ref{thm:Michael-Simon Allard}, we establish the following result.

\begin{theorem}
\label{thm:extinction time}
Let~$n\geq 1$,~$\alpha\in (0,1)$, and~$\Omega_0\subset\R^{n+1}$ be a bounded open convex set with~$C^2$ boundary. Let~$\left\{{\delomega}_t\right\}$ be the flow of hypersurfaces moving by fractional mean curvature~$H_{\alpha}.$ Then, the maximal time~$T^{\ast}$ defined by \eqref{max-time} satisfies
\begin{equation} \label{eq:boundforTast}
  T^{\ast}\leq C \hspace{1pt} |\partial\Omega_0|^{\frac{1+\alpha}{n}},
\end{equation}
for some constant~$C$ depending only on~$n$ and~$\alpha.$
\end{theorem}

The corresponding estimate for the classical mean curvature flow (where~$\alpha = 1$) was established by Evans and Spruck~\cite{Evans-Spruck}---see Evans~\cite[Section F.2]{Evans CIME} for a simpler proof in the case of a smooth flow. Both arguments make crucial use of the Michael-Simon inequality.

We stress that Theorem~\ref{thm:extinction time} assumes $\partial\Omega_t$ to be a~$C^2$ hypersurface for all~$t\in[0,T^\ast)$. Hence, our result must be understood as an estimate for the maximal time of existence of the $C^{2}$ flow, and not as a bound on the true extinction time~$T_e$. As commented in footnote~\ref{TastTefoot}, it is still not known whether for a convex~$C^2$ initial surface~$\partial\Omega_0$ the flow remains~$C^2$ for all times prior to extinction and there is no formation of singularities, such as, for instance, corners or edges of a polytope.

Note that Theorem~\ref{thm:extinction time} improves estimate~\eqref{SaezVal_bound} from~\cite{Saez Valdinoci} (when restricted to convex evolutions) in the dependence on~$\Omega_0$. Indeed, any bounded convex set~$\Omega_0$ satisfies the nontrivial inequality~$|\partial \Omega_0| \le C(n) \diam(\Omega_0)^n$---see Proposition~\ref{RSineprop} below for its sharp version, in which~$C(n)=2^{-n}|\partial B_1|=2^{-n}|\mathbb{S}^n|$. On the other hand, for~$n \ge 2$ one can produce examples of convex sets with diameter equal to~$1$ and arbitrarily small surface area---e.g., shrinking tubular neighborhoods of a segment.

\subsection{Sketch of the proof of Theorem~\ref{thm:intro:Convex MS for functions}}
\label{sketchsubsec}

Our analysis stems from the following observation. If~$\Omega\subset\R^{n+1}$ is an open convex set, then
\begin{equation}
 \label{eq:intro:pointwise ineq delomega}
 |\delomega|^{- \frac{\alpha}{n}} \leq CH_{\alpha}(x)\quad \text{for a.e.~$x\in\delomega$},
\end{equation}
for some universal constant~$C$ depending only on~$n$ and~$\alpha.$ This pointwise inequality for convex sets cannot hold for general domains---after replacing $H_\alpha$ by~$|H_\alpha|$. Indeed, one can easily construct a smooth bounded domain $\Omega$ with $0\in\partial\Omega$ and $H_{\alpha}(0)=0$; we will then have $|H_{\alpha}|\leq \varepsilon$ in a set of positive measure (a small neighborhood of $0$ on $\partial\Omega$), for every $\varepsilon>0$. It also has no counterpart in the local setting, since~$\partial \Omega$ may have flat parts where the standard mean curvature vanishes. The proof of~\eqref{eq:intro:pointwise ineq delomega} will be rather simple but, in any case, at the end of this section we will discuss how we originally found it in the plane, that is, when~$\Omega\subset\R^2$.

The next step in proving the main theorem is to consider subsets~$E\subset\partial\Omega$ and a dichotomy argument. 
We will distinguish, vaguely speaking, between two situations: either~$\partial \Omega$ has, at some well-chosen scales depending on~$|E|$, small density around~$x$, or it does not.
In the former case of points~$x$ of low density---occurring, say, where~$\partial \Omega$ has a tentacle-like shape---the proof of~\eqref{eq:intro:pointwise ineq delomega} still carries through and one obtains that
$$
  |E|^{-\frac{\alpha}{n}}\leq C H_{\alpha}(x)
$$
at such points~$x$. In the latter case when~$x$ has high density, 
we take advantage of the other term in the right-hand side of our fractional Michael-Simon inequality (in this exposition we take~$s=\alpha$ to simplify) and prove that
$$
  |E|^{-\frac{\alpha}{n}}\leq C\int_{\delomega\setminus E}\frac{dy}{|x-y|^{n+\alpha}}.
$$
This second case is what happens for nonlocal minimal surfaces
(Theorem~\ref{thm:intro:Cabre Cozzi fract Sobolev}), where every point has high density.
Either way, the following pointwise inequality will hold true: 
\begin{equation}\label{eq:intro:main pointwise ineq}
  |E|^{- \frac{\alpha}{n}} \le C \left( \int_{\delomega \setminus E} \frac{dy}{|x - y|^{n + \alpha }}+ H_\alpha(x) \right)\quad\text{for every  $E\subset\delomega$ and a.e.~$x\in E$}.
\end{equation}

Inequality~\eqref{eq:intro:main pointwise ineq}---see Proposition~\ref{pointwiseconvexMSprop}---is the key step towards the main theorem. It turns out to be the non-flat version of the pointwise estimate established in~$\R^n$ by Savin and Valdinoci~\cite[Appendix~A]{SV14}. Their estimate states that
\begin{equation}
 \label{eq:intro:pointwise by Savin Valdinoci}
  |E|^{-\frac{\alpha}{n}}\leq C \int_{\R^n \setminus E}\frac{dy}{|x-y|^{n+\alpha}} \quad\text{for all~$x \in \R^n$ and~$E\subset\R^n$},
\end{equation}
for some constant~$C$ depending only on~$n$ and~$\alpha$. This is a rearrangement inequality that follows immediately from the observation that integrating over the complement of the ball~$B_{\rho}(x)$, with~$|B_{\rho}|=|E|$, instead of~$\R^n \setminus E$ does not increase the right-hand side of~\eqref{eq:intro:pointwise by Savin Valdinoci}. 

Integrating~\eqref{eq:intro:main pointwise ineq} over~$E$, we are led to the fractional isoperimetric inequality
$$
  |E|^{\frac{n-\alpha}{n}}\leq C \left(\Per_{\delomega,\alpha}(E)+\int_{E}H_{\alpha}(x) \, dx\right).
$$
This is our fractional Sobolev inequality~\eqref{eq:intro:sAMSonconvex} for~$s = \alpha$,~$p=1$, and characteristic functions---i.e., inequality~\eqref{old theorem frac MS for sets}. To extend it to any~$p\geq 1$ and arbitrary functions we follow the strategy devised by Di Nezza, Palatucci, and Valdinoci in~\cite[Section~6]{HitchhikersG} to deduce the Euclidean fractional Sobolev inequality~\eqref{eq:intro:frac Sobolev euclidean}
from the pointwise inequality~\eqref{eq:intro:pointwise by Savin Valdinoci}.
As we will see later, when~$p = 1$ a fractional Sobolev inequality can be established more easily using the corresponding fractional isoperimetric inequality in combination with the fractional coarea formula of Visintin~\cite{Visintin}---see Lemma~\ref{lem:frac coarea on manifold} below. This is true both in the Euclidean framework and in the context of hypersurfaces. However, to the best of our knowledge, it is not known whether one can then derive the fractional Sobolev inequality for~$p > 1$ from the case~$p=1$ (even in the Euclidean case), in contrast with the case of Sobolev inequalities of integer order.

Finally, the following is an elementary proof of the pointwise lower bound~\eqref{eq:intro:pointwise ineq delomega} on the nonlocal mean curvature
for bounded and strictly convex sets of $\R^2.$ This was the starting point of our work. Up to a rigid movement, we may assume that~$x=0\in\partial\Omega$ and that~$\Omega \subset \R^2_+ = \{ y \in \R^2 : y_2 > 0 \}$. As~$\Omega$ is strictly convex, it can be parametrized by a function~$y:[0,\pi]\to \partial\Omega \subset \R^2$ of the form~$y(\theta)=r(\theta)(\cos\theta,\sin\theta)$, with~$r > 0$ in~$(0, \pi)$ and~$r(0)=r(\pi)=0$. In this parametrization, for~$y=y(\theta)$ and~$r=r(\theta)$, it holds
\begin{equation}
 \label{eq:intro:kernel of double layer pot}
  y\cdot\nu(y)=\frac{r^2}{\sqrt{r^2+\dot{r}^2}}
  \quad\text{ and }\quad 
  \frac{y\cdot\nu(y)}{|y|^{2+\alpha}}=\frac{1}{r^{\alpha}\sqrt{r^2+\dot{r}^2}}.
\end{equation}
From the fact that~$r/\sqrt{r^2+\dot{r}^2}\leq 1$ one obtains
\begin{equation} \label{eq:intro:idea Gauss Bonnet}
\begin{aligned}
  \pi & = \int_0^{\pi}d\theta
  \leq
  \int_0^{\pi}\left(\frac{r}{\sqrt{r^2+\dot{r}^2}}\right)^{-\frac{\alpha}{1+\alpha}}  \, d\theta
  = \int_0^{\pi} \! r^{-\frac{\alpha}{1+\alpha}} \!
  \left(\sqrt{r^2+\dot{r}^2}\right)^{\frac{\alpha}{1+\alpha}} \, d\theta 
  \\
  & \leq
  \left(\int_0^{\pi}\frac{d\theta}{r^{\alpha}}\right)^{\!\! \frac{1}{1+\alpha}} \! \left(\int_0^{\pi} \! \sqrt{r^2+\dot{r}^2} \, d\theta\right)^{\!\! \frac{\alpha}{1+\alpha}}.
\end{aligned}
\end{equation}
From the representation~\eqref{eq:intro:def of H alpha} for the fractional~$\alpha$-mean curvature and the second identity in~\eqref{eq:intro:kernel of double layer pot} it follows that $\int_0^\pi r^{-\alpha}d\theta=H_{\alpha}(0).$
Hence we proved that~$\pi^{1+\alpha}~\leq~H_{\alpha}(0)|\partial\Omega|^{\alpha}$, which is precisely~\eqref{eq:intro:pointwise ineq delomega} for~$n=1$.

\subsection{Plan of the Paper}

We shall prove Theorem~\ref{thm:intro:Convex MS for functions} in increasing order of generality, using in each section the previous less general results or the main ingredients of their proofs.

In Section~\ref{section:frac Aleks Fench} we prove the pointwise lower bound~\eqref{eq:intro:pointwise ineq delomega} for~$H_{\alpha},$ as well as its integral consequence~\eqref{old thm frac Aleks-Fench}. This is the last statement of Theorem~\ref{thm:intro:Convex MS for functions}.
	
In Section~\ref{section:subsets of convex set} we extend the pointwise inequality~\eqref{eq:intro:pointwise ineq delomega} to proper subsets~$E$ of~$\delomega$---i.e., we prove the pointwise lower bound~\eqref{eq:intro:main pointwise ineq}.

In Section~\ref{section:fromsetstofunctions} we deduce Theorem~\ref{thm:intro:Convex MS for functions} in its full generality from the pointwise lower bound~\eqref{eq:intro:main pointwise ineq}.

In Section~\ref{section:fractional mean curv flow} we apply the pointwise inequality~\eqref{eq:intro:pointwise ineq delomega} to the fractional mean curvature flow and establish Theorem~\ref{thm:extinction time}.

In Appendix~\ref{appendix:rosenthal-szasz} we provide a simple proof of a known isodiametric inequality for the perimeter of convex sets.

\subsection{Notation}
Throughout the paper, the word \emph{measurable} refers to the~$n$-dimensional Hausdorff measure~$\mathcal{H}^n$ on a hypersurface~$M$ of~$\R^{n+1}$, if not stated explicitly otherwise. The measure of a set~$E \subset M$ will be denoted by~$|E|$ and the integration element simply by~$dx$, instead of $d\Haus^n(x).$ Open balls are understood as balls in the ambient space~$\R^{n+1},$ i.e.,~$B_R(x)=\{y\in \R^{n+1}:\,|y-x|<R\}$ and~$|y-x|$ is the Euclidean distance in~$\R^{n+1}.$ If~$x=0$ we write~$B_R=B_R(0),$ while~$\mathbb{S}^n = \partial B_1$ is the~$n$-dimensional unit sphere in~$\R^{n+1}.$

\section{A lower bound on the nonlocal mean curvature}
\label{section:frac Aleks Fench}

\noindent
We start by proving the last bound~\eqref{old thm frac Aleks-Fench} of Theorem~\ref{thm:intro:Convex MS for functions}, which is the simplest statement within the theorem. It will follow from the pointwise inequality
\begin{equation}\label{eq:old pointwise aleks Fench}
 |\partial\Omega|^{-\frac{\alpha}{n}}\leq C\, H_{\alpha}(x) \quad\text{for a.e.~}x\in
  \partial\Omega,\quad\text{where }C=\left(\frac{2}{|\mathbb{S}^n|}\right)^{\frac{n+\alpha}{n}}.
\end{equation}
Here,~$\alpha \in (0,1)$ and~$\Omega$ is any bounded open convex set of~$\R^{n+1}$.

The proof of~\eqref{eq:old pointwise aleks Fench} relies on the following two simple lemmas. The first one extends the first identity in~\eqref{eq:intro:idea Gauss Bonnet} to higher dimensions. It is a well-known result in the theory of double layer potentials (it is sometimes called Gauss law) and does not require convexity; see, e.g.,~\cite[Proposition 3.19]{Folland PDE}. Later we will use the lemma with~$\Omega_b=\Omega\cap B_R(x)$ for some radius~$R$, where~$\Omega$ is our convex set and~$x\in\partial\Omega$.

\begin{lemma} \label{fundsollem}
Let~$\Omega_b \subset \R^{n + 1}$ be a bounded domain with Lipschitz boundary. Then,
$$
  \PV\int_{\delomega_b} \frac{(y - x) \cdot \nu(y)}{|y - x|^{n + 1}} \, {dy} = \frac{|\mathbb{S}^n|}{2}
$$
holds true in the principal value sense at every point~$x \in \delomega_b$ at which~$\delomega_b$ is differentiable.
\end{lemma}

\begin{proof}
Let~$\Phi$ be the fundamental solution of the Laplacian centered at~$x$, i.e.,~$- \Delta \Phi = \delta_x$ in $\R^{n+1}$. We have that $\nabla \Phi(y) = - |\mathbb{S}^n|^{-1} |y - x|^{- n - 1} (y - x)$ for every~$y \ne x$. Let~$\varepsilon > 0$ be sufficiently small. By applying the divergence theorem in~$\Omega_b \setminus \overline{B}_\varepsilon(x)$, a Lipschitz domain, we get that
\begin{align*}
  0 & = \int_{\Omega_b \setminus \overline{B}_\varepsilon(x)} \Delta \Phi(y) \, dy = \int_{\partial \left( \Omega_b \setminus \overline{B}_\varepsilon(x) \right)} \nabla \Phi(y) \cdot \nu(y) \, {dy}
  \\
  & = \frac{1}{|\mathbb{S}^n|} \left\{ - \int_{{\delomega_b} \setminus \overline{B}_\varepsilon(x)} \frac{(y - x) \cdot \nu(y)}{|y - x|^{n + 1}} \, {dy} + \frac{|\Omega_b \cap \partial B_\varepsilon(x)|}{\varepsilon^n} \right\}.
\end{align*}
The claim follows by letting~$\varepsilon \rightarrow 0^+$ and noticing that~$\varepsilon^{- n} |\Omega_b \cap \partial B_\varepsilon(x)| \rightarrow |\mathbb{S}^n| / 2$, since $\partial\Omega$ is differentiable at $x.$ This shows in particular that the principal value in the statement exists.
\end{proof}

The second lemma is an extension of the inequalities in~\eqref{eq:intro:idea Gauss Bonnet} to any dimension~$n \ge 1$. For the proof of~\eqref{eq:old pointwise aleks Fench} we will only need the next lemma for~$E=\delomega,$ but in Section~\ref{section:subsets of convex set} we will require the estimate for general subsets~$E\subset\delomega$. Here it is useful to recall the comments made before Theorem~\ref{thm:intro:Convex MS for functions} on the differentiability properties of open convex sets~$\Omega$ and the definition~\eqref{eq:intro:def of H alpha} of~$H_\alpha(x)$ for~$x\in\partial\Omega$.

\begin{lemma} \label{mainHalphaestlem}
Let~$\alpha\in (0,1)$ and $\Omega\subset\R^{n+1}$ be an open convex set. Then,
$$
  \int_E \frac{(y - x) \cdot \nu(y)}{|y - x|^{n + 1}} \, {dy} \le |E|^{\frac{\alpha}{n + \alpha}} H_{\alpha}(x)^{\frac{n}{n + \alpha}}
$$
for every measurable subset~$E\subset\delomega$ and almost every point~$x\in \delomega$. Here, the integral is well-defined in~$[0,+\infty]$ since its integrand is a non-negative function.

\end{lemma}

\begin{proof}
Using that~$0 \le (y - x) \cdot \nu(y) / |y - x| \le 1$ for almost every~$x$ and~$y$ on~$\partial\Omega$, together with H\"older's inequality, we see that
\begin{align*}
 \int_E \frac{(y - x) \cdot \nu(y)}{|y - x|^{n + 1}} \, {dy} 
 & =
 \int_E \frac{(y - x) \cdot \nu(y)}{|y - x|} \frac{{dy}}{|y - x|^n} 
 \\
 & \le
  \int_E \left( \frac{(y - x) \cdot \nu(y)}{|y - x|} \right)^{\frac{n}{n + \alpha}} \frac{{dy}}{|y - x|^n}  
  = 
  \int_E \left( \frac{(y - x) \cdot \nu(y)}{|y - x|^{n + 1 + \alpha}} \right)^{\frac{n}{n + \alpha}} {dy} 
  \\
 & \le
  |E|^{\frac{\alpha}{n + \alpha}} \left( \int_E \frac{(y - x) \cdot \nu(y)}{|y - x|^{n + 1 + \alpha}} \, {dy} \right)^{\frac{n}{n + \alpha}}.
\end{align*}
The claim follows from this inequality and expression~\eqref{eq:intro:def of H alpha} for~$H_\alpha(x)$, combined with the fact that~$(y - x) \cdot \nu(y) \ge 0$ for almost every~$x$ and~$y$ on~$\delomega$ since~$\Omega$ is convex.
\end{proof}

\begin{proof}[Proof of inequalities \eqref{eq:old pointwise aleks Fench} and \eqref{old thm frac Aleks-Fench}.]
Since Lemma~\ref{fundsollem} requires the domain to be boun\-ded, while \eqref{old thm frac Aleks-Fench} is claimed for convex sets with finite perimeter, we first point out that an open convex set is bounded if and only if it has finite perimeter\footnote{This follows from the monotonicity of the perimeter of open convex sets with respect to inclusion (see Appendix~\ref{appendix:rosenthal-szasz} for the proof of this classical result). Using this fact, one part of the claim is obvious, while the other is checked as follows. Any unbounded open convex set contains a ball, and hence also the convex cones generated by a vertex going to infinity and the ball. Note finally that such convex cones have arbitrarily large perimeter.}.

Now, from Lemmas~\ref{fundsollem} and~\ref{mainHalphaestlem} applied with~$\Omega_b = \Omega$ and~$E=\delomega$, respectively, we obtain~\eqref{eq:old pointwise aleks Fench}. By raising~\eqref{eq:old pointwise aleks Fench} to the power~$s/\alpha$ and integrating over~$\delomega,$ we infer the validity of~\eqref{old thm frac Aleks-Fench}.
\end{proof}

\section{A fractional Michael-Simon type isoperimetric inequality}
\label{section:subsets of convex set}

\noindent
In this section we shall prove the key pointwise inequality involved in the proof of Theorem~\ref{thm:intro:Convex MS for functions}. This is the content of the following proposition.

\begin{proposition} \label{pointwiseconvexMSprop}
Let~$\alpha\in(0,1),$ $s\in(0,1),$ $p > 0$, and~$\Omega \subset \R^{n + 1}$ be an open convex set. Then, for every~$E \subset \delomega$ with finite positive measure and a.e.~$x \in E$, it holds
\begin{equation} 
 \label{pointwiseconvexMSine}
  |E|^{- \frac{s p}{n}} \le C \left( \int_{\delomega \setminus E} \frac{{dy}}{|x - y|^{n + s p}} + H_{\alpha}(x)^{\frac{s p}{\alpha}}\right)
\end{equation}
for some constant~$C$ depending only on~$n$,~$\alpha$,~$s$,~and~$p$.
\end{proposition}

The proof of this result relies on the following ingredients: 
\begin{enumerate}[label=($\roman*$),leftmargin=28pt]
\item \label{(i)} A classical rearrangement result, Lemma~\ref{lemma:E can be delomega cap BR}, which reduces the proof of~\eqref{pointwiseconvexMSine} for a general set~$E$ to the case~$E=\delomega\cap B_R(x).$
 
\item \label{(ii)} The double layer potential identity of Lemma~\ref{fundsollem}, and the localized pointwise inequality that will follow from it, Lemma~\ref{lemma:localized double layer pot}.
 
\item \label{(iii)} A dichotomy argument. We will essentially distinguish between two cases, depending on whether~$R^{-n}|\delomega \cap B_R(x)|$ is smaller or larger than a certain constant for some appropriate radii R which depend on $|E|.$

\item \label{(iv)} A kind of reverse  perimeter-energy estimate for convex sets, Lemma~\ref{isoplem} (\textit{reverse} here is meant in comparison with the natural upper bound on the perimeter that holds for minimizing minimal surfaces).

\item \label{(v)} A known isodiametric inequality for the perimeter of convex sets, Proposition~\ref{RSineprop}.
 
\end{enumerate}

We start with~\ref{(i)}---the rearrangement result---which is the content of the next lemma. From it, it will be enough to prove Proposition~\ref{pointwiseconvexMSprop} for sets of the form~$E=\partial\Omega\cap B_R(x)$. Indeed, the lemma states that replacing~$E$ by~$\partial\Omega\cap B_R(x)$, with~$|E| = |\delomega \cap B_R(x)|$, does not increase the right-hand side of~\eqref{pointwiseconvexMSine}. This elementary observation does not require any convexity assumption on~$\Omega$, nor that the hypersurface~${\delomega}$ is a boundary in the first place. In addition, all what is needed for the exponent $n+sp$ appearing in the statement is to be larger than $n.$ This is why we allow $p>0$ in the lemma---though we will use it always with $p\geq 1.$

\begin{lemma}
\label{lemma:E can be delomega cap BR}
Let~$s \in (0,1)$,~$p\in (0,+\infty)$, and~$M\subset\R^{n+1}$ be a set of locally finite~$n$-dimensional Hausdorff measure. Let~$E\subset M$ be a set with positive measure and~$x\in E$. Assume that
\begin{equation}
 \label{eq:assumption E is delomega cap BR}
  |E|=|{M}\cap B_R(x)|
\end{equation}
for some $R>0.$

Then,
$$
  \int_{{M}\setminus E}\frac{dy}{|x-y|^{n+s p}}
  \geq
  \int_{{M}\setminus B_R(x)}\frac{dy}{|x-y|^{n+s p}}.
$$
\end{lemma}

\begin{proof}
By assumption~\eqref{eq:assumption E is delomega cap BR} we have 
$$
|E \cap B_R(x)| + |({M}\setminus E)\cap B_R(x)| = |E \cap B_R(x)| + |E \setminus B_R(x)|,
$$
and thus
$$
  |({M}\setminus E)\cap B_R(x)| = |E \setminus B_R(x)|. 
$$
By using this identity we see that
\begin{align*}
  \int_{{M} \setminus E} \frac{{dy}}{|y - x|^{n + s{p}}} & 
  = \int_{\left( {M} \setminus E \right) \cap B_R(x)} \frac{{dy}}{|y - x|^{n + s{p}}} + \int_{({M} \setminus E) \setminus B_R(x)} \frac{{dy}}{|y - x|^{n + s{p}}} 
  \\
  & \ge R^{- n - s{p}} |({M}\setminus E)\cap B_R(x)| + \int_{({M} \setminus E) \setminus B_R(x)} \frac{{dy}}{|y - x|^{n + s{p}}} 
  \\
  & = R^{- n - s{p}} |E \setminus B_R(x)| + \int_{({M} \setminus E) \setminus B_R(x)} \frac{{dy}}{|y - x|^{n + s{p}}}
  \\
  & \ge \int_{E \setminus B_R(x)} \frac{{dy}}{|y - x|^{n + s{p}}} + \int_{({M} \setminus E) \setminus B_R(x)} \frac{{dy}}{|y - x|^{n + s{p}}} 
  \\
  & = \int_{{M} \setminus B_R(x)} \frac{{dy}}{|y - x|^{n + s{p}}}.
\end{align*}
This proves the lemma.
\end{proof}

The next lemma is of technical nature and we will use it twice. Within the proof of Proposition~\ref{pointwiseconvexMSprop} it will guarantee that, under appropriate assumptions on~$M$ and for almost every~$x\in E\subset M$, hypothesis~\eqref{eq:assumption E is delomega cap BR} is actually satisfied for some radius~$R$ depending on~$x$---a property that may not be satisfied by all~$x\in E$, as we will see.

\begin{lemma} \label{Axcontlem}Let~$M \subset \R^{n + 1}$ be a set of locally finite~$n$-dimensional Hausdorff measure. Then, the following statements hold true.
\begin{enumerate}[label=$(\alph*)$,leftmargin=*]
\item \label{Discountable} The set
$$
D := \Big\{ x \in M : \mbox{there exists } R_x > 0 \mbox{ such that } \left| M \cap \partial B_{R_x}(x) \right| > 0 \Big\}
$$
is at most countable.
\item \label{Axprops} For every~$x \in {M}$, the function
$$
\A_x: (0, +\infty) \to [0, +\infty), \mbox{ defined by } \A_x(R) := |{M} \cap B_R(x)| \mbox{ for } R > 0,
$$
is non-decreasing and continuous from the left. Furthermore, it is continuous if and only if~$x \in M \setminus D$.
\end{enumerate}
\end{lemma}

\begin{proof}
Consider, for~$j, k \in \N$, the sets~$M_j := M \cap B_j$ and
$$
D_{j, k} := \left\{ x \in M_j : \mbox{there exists } R_x > 0 \mbox{ such that } |M_j \cap \partial B_{R_x}(x)| \ge \frac{1}{k} \right\}.
$$
It is clear that~$D = \cup_{j, k \in \N} D_{j, k}$.
Note that, if~$x$ and~$y$ are two distinct points in~$D_{j, k}$, then~$\left| \big( M_j \cap \partial B_{R_x}(x) \big) \cup \big( M_j \cap \partial B_{R_y}(y) \big) \right| = 0$. Moreover, as~$M$ has locally finite~$n$-dimensional measure, we have that~$|M_j| < +\infty$. From the last two facts we deduce that each~$D_{j, k}$ contains no more than~$k |M_j|$ points. Hence,~$D$ is at most countable and~\ref{Discountable} is proved.

We now address point~\ref{Axprops}. The monotonicity of the function~$\A_x$ is obvious, while its left-continuity follows from~$B_R(x)$ being open. The last statement is a consequence of the fact that~$|{M} \cap \partial B_R(x)|= \lim_{\rho \rightarrow R^+} \A_x(\rho) - \A_x(R)$ for every~$x \in M$ and~$R > 0$. We stress that for the last two claims we took advantage of the~$\Haus^n$-measurability of~$M$ and of standard formulas for the measure of increasing unions and decreasing intersections of sets.
\end{proof}

In the following result we apply the double layer potential identity of Lemma~\ref{fundsollem} with~$\Omega_b=\Omega\cap B_R(x)$. This allows us to obtain a localized version of Lemma~\ref{mainHalphaestlem}. 

\begin{lemma}
\label{lemma:localized double layer pot}
Let~$\alpha\in(0,1)$ and~$\Omega\subset\R^{n+1}$ be an open convex set. Then,
\begin{align*}
  \frac{|\mathbb{S}^n|}{2}
  & = 
  \int_{\partial\Omega\cap B_R(x)}\frac{(y-x)\cdot \nu(y)}{|y-x|^{n+1}} \, dy
  +
  \frac{|\Omega\cap \partial B_R(x)|}{R^n}
  \smallskip 
  \\
  & \leq
  |\partial\Omega\cap B_R(x)|^{\frac{\alpha}{n+\alpha}}H_{\alpha}(x)^{\frac{n}{n+\alpha}}+\frac{|\Omega\cap \partial B_R(x)|}{R^n}
\end{align*}
for every~$R>0$ and almost every~$x\in\partial\Omega.$
\end{lemma}

\begin{proof}
First, recall that, as~$\Omega$ is convex, its boundary is Lipschitz and has therefore locally finite~$n$-dimensional Hausdorff measure. Hence, we may apply Lemma~\ref{Axcontlem}\ref{Discountable} and deduce that, for all but a countable number of points~$x \in \partial \Omega$, it holds~$|\partial \Omega \cap \partial B_R(x)| = 0$ for every~$R > 0$. Moreover,~$\partial \Omega$ is differentiable at almost all of such points.

Consider the convex set~$\Omega_b=\Omega\cap B_R(x)$. Its boundary~$\partial\Omega_b$ is therefore Lipschitz and, in addition, it is equal, up to a set of measure zero, to the disjoint union of the two sets~$\partial\Omega\cap B_R(x)$ and~$\Omega\cap\partial B_R(x)$---we used here the fact, noted earlier, that~$|\partial \Omega \cap \partial B_R(x)| = 0$. Applying the double layer potential identity of Lemma \ref{fundsollem}, we get
\begin{align*}
  \frac{|\mathbb{S}^n|}{2}
  =&
   \int_{\partial \Omega_b}\frac{(y-x)\cdot\nu(y)}{|y-x|^{n+1}} \, dy
  \smallskip
  \\
  =& \int_{\partial \Omega\cap B_R(x)}\frac{(y-x)\cdot\nu(y)}{|y-x|^{n+1}} \, dy
  +\int_{\Omega\cap \partial B_R(x)}\frac{(y-x)\cdot\nu(y)}{|y-x|^{n+1}} \, dy.
\end{align*}
As~$(y-x)\cdot\nu(y)=|y-x|=R$ for all~$y$ on~$\partial B_R(x),$ the first identity in the lemma is proved. The second inequality follows from Lemma~\ref{mainHalphaestlem}, applied to the set~$E=\partial\Omega\cap B_R(x).$
\end{proof}

When~$\partial\Omega$ has low density around a point~$x$ at a certain scale, we will absorb the last term in the inequality of Lemma \ref{lemma:localized double layer pot} within its left-hand side. For this we will need the following reverse perimeter-energy estimate.

\begin{lemma} \label{isoplem}
Let~$\Omega \subset \R^{n + 1}$ be an open convex set,~$x \in \partial \Omega,$ and $R>0.$ Then,
$$
  |\Omega \cap \partial B_R(x)| \le \frac{C_n}{R} \, |\delomega \cap B_R(x)|^{\frac{n + 1}{n}}
$$
for some constant $C_n\geq 1$ depending only on $n.$
\end{lemma}

\begin{proof}
Of course, we can assume that~$\Omega \not\subset B_R(x)$, since otherwise there is nothing to prove. Let~$\CC$ be the open cone of vertex~$x$ spanned by~$\Omega \cap \partial B_R(x)$. By the coarea formula  and the homogeneity of cones (here~$\mathcal{H}^{n+1}$ denotes the~$(n+1)$-dimensional Lebesgue measure in~$\R^{n+1}$),
$$
  \mathcal{H}^{n+1}(\CC \cap B_R(x)) = \int_0^R |\CC \cap \partial B_\rho(x)| \, d\rho = \frac{|\CC \cap \partial B_R(x)|}{R^n} \int_0^R \rho^n \, d\rho = \frac{|\CC \cap \partial B_R(x)| R}{n + 1}.
$$
Moreover, as~$\Omega \cap B_R(x)$ is convex, we have that~$\CC \cap B_R(x) \subset \Omega \cap B_R(x)$. Consequently,
\begin{align*}
  |\Omega \cap \partial B_R(x)| &= |\CC \cap \partial B_R(x)| = (n + 1) \frac{\mathcal{H}^{n+1}(\CC \cap B_R(x))}{R} \\
  & \le (n + 1) \frac{\mathcal{H}^{n+1}(\Omega \cap B_R(x))}{R}.
\end{align*}

Now, by the relative isoperimetric inequality in Euclidean balls (see, e.g.,~\cite[Proposition~12.37 and Remark~12.38]{M12}),
$$
  \min\left\{ \mathcal{H}^{n+1}(\Omega\cap B_R(x)),
  \mathcal{H}^{n+1}(B_R(x)\setminus\Omega)\right\}
  \leq C_I|\delomega\cap B_R(x)|^{\frac{n+1}{n}}
$$
for some constant~$C_I$ depending only on~$n$. Using again the convexity of~$\Omega$ to ensure that the minimum on the left-hand side is~$\mathcal{H}^{n+1}(\Omega \cap B_R(x))$, we conclude that
$$
  |\Omega \cap \partial B_R(x)| \le \frac{C_n}{R} \, |{\delomega} \cap B_R(x)|^{\frac{n + 1}{n}},
$$
where~$C_n= (n + 1) C_I$.
\end{proof}

To deal with the second case in the dichotomy~\ref{(iii)}, where the point~$x\in\delomega$ has large density for some radii~$R,$ we will need the following isodiametric inequality for the perimeter of convex sets.

\begin{proposition}[Rosenthal-Sz\'asz type inequality; see, e.g.,~\cite{BF87}] \label{RSineprop}
Let~$\Omega \subset \R^{n + 1}$ be a bounded open convex set. Then,
\begin{equation}
 \label{eq:Rosenthal szasz ineq in Prop}
\frac{|\delomega|}{\diam(\Omega)^n} \le \frac{|\partial B_1|}{\diam(B_1)^n} = \frac{|\mathbb{S}^n|}{2^n}.
\end{equation}

As a consequence,
\begin{equation}
 \label{eq:perestcor}
  |\delomega \cap B_R(x)| \le |\mathbb{S}^n| R^n
\end{equation}
for every open convex set~$\Omega \subset \R^{n + 1}$,~$x\in\delomega$, and~$R>0.$
\end{proposition}

The Rosenthal-Sz\'asz inequality~\eqref{eq:Rosenthal szasz ineq in Prop} is classical and probably well-known to expert readers. It is stated in Section~44 of~\cite{BF87} as inequality~(6) and proved in that monograph throughout several sections. Since we could not find a reference with a short proof of the inequality, we will include it in Appendix~\ref{appendix:rosenthal-szasz}. Estimate~\eqref{eq:perestcor} is immediately deduced by applying~\eqref{eq:Rosenthal szasz ineq in Prop} to the bounded open convex set~$\Omega \cap B_R(x)$ and using that~$\partial \Omega \cap B_R(x) \subset \partial ( \Omega \cap B_R(x) )$.

Observe that~\eqref{eq:Rosenthal szasz ineq in Prop} carries the optimal constant.
For our purposes, we only need~\eqref{eq:perestcor}, and we do not need it with its best constant. That is, we will only use that $|\delomega\cap B_R(x)|\leq C R^n,$ for some dimensional constant $C,$ for every open convex set $\Omega$. We also include in Appendix~\ref{appendix:rosenthal-szasz} a simple proof of this non-optimal inequality.

We have now all the preliminary results to prove Proposition~\ref{pointwiseconvexMSprop}.

\begin{proof}[Proof of Proposition~\ref{pointwiseconvexMSprop}.]
Without loss of generality we can take~$E$ to be bounded, by proving the proposition first for~$E_k := E\cap B_{k}(x),$ $k\in\mathbb{N}$, and then letting~$k\to+\infty$. To this aim, notice that
$$
\left| \int_{\delomega \setminus E_k} \frac{{dy}}{|x - y|^{n + s p}} - \int_{\delomega \setminus E} \frac{{dy}}{|x - y|^{n + s p}} \right| = \int_{E \setminus B_k(x)} \frac{{dy}}{|x - y|^{n + s p}} \le \frac{|E|}{k^{n + s p}} \longrightarrow 0
$$
as~$k \rightarrow +\infty$, since~$|E| < + \infty$.

Since now~$E$ is bounded, we can assume, using Lemmas~\ref{lemma:E can be delomega cap BR} and~\ref{Axcontlem}, that~$E$ is of the form
\begin{equation} \label{Eisround}
E=\delomega\cap B_R(x),
\end{equation}
for some~$R > 0$. Indeed, by Lemma~\ref{Axcontlem}\ref{Axprops}, the function~$R \mapsto |\partial \Omega \cap B_R(x)|$ is continuous for almost every~$x \in \partial \Omega$. Thus, we can clearly choose a radius~$R>0,$ depending on~$x,$ such that~$|\partial \Omega \cap B_R(x)|=|E|$. Now, Lemma~\ref{lemma:E can be delomega cap BR} says that, replacing~$E$ by~$\delomega \cap B_R(x),$ the right-hand side of~\eqref{pointwiseconvexMSine} does not increase, while its left-hand side remains unaltered. We can therefore take~$E$ to be given by~\eqref{Eisround}. 

We now distinguish between three cases, involving different assumptions on the density of~$\partial\Omega$ around~$x$. We will compare the density~$\rho^{-n}|\partial\Omega\cap B_\rho(x)|$ with the dimensional constant
$$
  \delta := \min\left\{|\mathbb{S}^n|,\left(\frac{|\mathbb{S}^n|}{4 C_n}\right)^{\frac{n}{n+1}}\right\}
$$
at the two different scales~$\rho = R$ and~$\rho = TR$, where~$C_n \ge 1$ is the constant from Lemma~\ref{isoplem} and
$$
T := \left( \frac{2 |\mathbb{S}^n|}{\delta} \right)^{\frac{1}{n}} > 1.
$$

\textit{Case 1.} Assume that
\begin{equation} \label{case1deltacond}
  |\delomega \cap B_R(x)| \le \delta R^n.
\end{equation}
Using Lemma~\ref{lemma:localized double layer pot}, we deduce that
$$
\frac{|\mathbb{S}^n|}{2} \leq |\partial\Omega\cap B_R(x)|^{\frac{\alpha}{n+\alpha}}H_{\alpha}(x)^{\frac{n}{n+\alpha}}+\frac{|\Omega\cap \partial B_R(x)|}{R^n}.
$$
We estimate the second term on the right with the aid of Lemma~\ref{isoplem}, assumption~\eqref{case1deltacond}, and the definition of~$\delta$, getting that
$$
\frac{|\Omega\cap \partial B_R(x)|}{R^n} \le C_n \delta^{\frac{n + 1}{n}} \le \frac{|\mathbb{S}^n|}{4}.
$$
The combination of the previous two inequalities leads us to
\begin{equation} \label{claimincase1}
\frac{|\mathbb{S}^n|}{4} \leq |\partial\Omega\cap B_R(x)|^{\frac{\alpha}{n+\alpha}}H_{\alpha}(x)^{\frac{n}{n+\alpha}},
\end{equation}
which, recalling~\eqref{Eisround}, establishes~\eqref{pointwiseconvexMSine} in this first case.

\textit{Case 2.} We assume now that
\begin{equation} \label{case2deltacond}
  |\delomega\cap B_R(x)|>\delta R^n \quad \mbox{ and } \quad |\delomega \cap B_{T R}(x)| \le \delta (T R)^n.
\end{equation}
Arguing exactly as for~\eqref{claimincase1}, but now with $R$ replaced by $TR$, we obtain that
$$
\frac{|\mathbb{S}^n|}{4} \leq |\partial\Omega\cap B_{T R}(x)|^{\frac{\alpha}{n+\alpha}}H_{\alpha}(x)^{\frac{n}{n+\alpha}}.
$$
The two inequalities in~\eqref{case2deltacond} give that~$|\delomega \cap B_{T R}(x)| \le T^n |\delomega\cap B_R(x)| = T^n |E|$. Hence,
$$
\frac{|\mathbb{S}^n|}{4} \leq T^{\frac{n \alpha}{n + \alpha}} |E|^{\frac{\alpha}{n + \alpha}} H_{\alpha}(x)^{\frac{n}{n+\alpha}},
$$
which yields~\eqref{pointwiseconvexMSine} with a new constant~$C$.

\textit{Case 3.} Finally, we assume that
\begin{equation*}
  \qquad |\delomega \cap B_R(x)| > \delta R^n\quad\text{ and }\quad
  |\delomega\cap B_{TR}(x)|> \delta (TR)^n.
\end{equation*}
Taking advantage of the perimeter bound~\eqref{eq:perestcor}, we see that
\begin{align*}
  |\delomega \cap \left( B_{T R}(x) \setminus B_R(x) \right)| & = |\delomega \cap B_{T R}(x)| - |\delomega \cap B_R(x)| \\
& \ge \delta (T R)^n - |\mathbb{S}^n| R^n = |\mathbb{S}^n| R^n.
\end{align*}
Consequently, we find that
\begin{align*}
  \int_{\delomega \setminus E} \frac{{dy}}{|y - x|^{n + s {p}}} & \ge \int_{\delomega \cap \left( B_{T R}(x) \setminus B_R(x) \right)} \frac{{dy}}{|y - x|^{n + s{p}}} \ge \frac{|\delomega \cap \left( B_{T R}(x) \setminus B_R(x) \right)|}{(T R)^{n + s{p}}} 
  \\
  & \ge \frac{|\mathbb{S}^n|}{T^{n + s{p}}} R^{- s{p}}\ge \frac{|\mathbb{S}^n| \delta^{\frac{s{p}}{n}}}{T^{n + s{p}}} \, |\delomega \cap B_R(x)|^{- \frac{s{p}}{n}} = \frac{|\mathbb{S}^n| \delta^{\frac{s{p}}{n}}}{T^{n + s{p}}} \, |E|^{- \frac{s{p}}{n}},
\end{align*}
which yields~\eqref{pointwiseconvexMSine} once again for some constant $C.$ 

As this was the last case, the proof of Proposition~\ref{pointwiseconvexMSprop} is finished.
\end{proof}

\section{Fractional Michael-Simon inequality for functions} \label{section:fromsetstofunctions}

\noindent
In this section we establish our main result, inequality~\eqref{eq:intro:sAMSonconvex} of Theorem~\ref{thm:intro:Convex MS for functions}. Namely, that for every measurable function~$u \in W^{s, p}(\delomega)$ it holds
\begin{equation}\label{eq:Sec 4 main inequal}
  \|{u}\|_{L^{p^{\ast}_s\!}({\delomega})}\leq C\left(\frac{1}{2}\int_{{\delomega}}\int_{{\delomega}} \frac{|{u}(x)-{u}(y)|^p}{|x-y|^{n+sp}} \,{dx}{dy}+\int_{{\delomega}}H_{\alpha}(x)^{\frac{s p}{\alpha}} |{u}(x)|^p \, dx
  \right)^{\frac{1}{p}},
\end{equation}
where~$C$ is a constant depending only on~$n$,~$\alpha$,~$s$, and~$p$.

We first give a proof when $p=1.$ This is simple and based on the fractional coarea formula of Visintin \cite{Visintin}. This first proof gives the same constant in \eqref{eq:Sec 4 main inequal} as the one in the isoperimetric inequality 
\eqref{old theorem frac MS for sets} of Theorem~\ref{thm:intro:Convex MS for functions}, which also agrees with the constant in the pointwise inequality of Proposition~\ref{pointwiseconvexMSprop}.

It is important to point out that in contrast with the local case, for $p>1$ it is not known how to derive a fractional Sobolev inequality from a corresponding fractional isoperimetric inequality, even in Euclidean space. Thus we give a second proof of our fractional Sobolev inequality that is valid for all~$p\geq 1.$
For~$p = 1$, it gives a worse constant than the one found via the coarea formula.
This second argument follows very closely the slicing procedure of Savin and Valdinoci~\cite{Savin Valdinoci 1} and Di Nezza, Palatucci, and Valdinoci~\cite[Section 6]{HitchhikersG}, with the necessary modifications to cope with the term involving~$H_{\alpha}$.

For the first proof, we will need the following version of the fractional coarea formula.

\begin{lemma}[Fractional coarea formula on manifolds]
\label{lem:frac coarea on manifold}
Let~${M}\subset \R^{n+1}$ be a Lipschitz hypersurface,~$s\in (0,1)$, and~${u}:{M}\to \R$ be a measurable function. Then
$$
  \frac{1}{2}\int_{{M}}\int_{{M}}\frac{|{u}(x)-{u}(y)|}{|x-y|^{n+s}} \, {dx}{dy}=\int_{-\infty}^{+\infty}\Per_{{M},s}(\{{u}>t\}) \, dt.
$$
\end{lemma}

\begin{proof}
Using the layer cake representation, one writes
\begin{equation}
 \label{eq:layer cake repr}
  {u}(x)-{u}(y)=\int_{-\infty}^{+\infty}\chi_{\{{u}>t\}}(x)\chi_{\{{u}\leq t\}}(y) \, dt,\quad\text{ if }{u}(x)>{u}(y).
\end{equation}
If~${u}(x)\leq {u}(y)$ then the right-hand side of~\eqref{eq:layer cake repr} vanishes. Therefore
$$
  |{u}(x)-{u}(y)|=\int_{-\infty}^{+\infty}\left(\chi_{\{{u}>t\}}(x)\chi_{\{{u}\leq t\}}(y)+\chi_{\{{u}\leq t\}}(x)\chi_{\{{u}>t\}}(y)
  \right)dt.
$$
Now, an application of Fubini theorem on~$\R\times {M}$ gives 
\begin{align*}
  \int_{{M}}\int_{{M}}\frac{|{u}(x)-{u}(y)|}{|x-y|^{n+s}} \, {dx}{dy}= 2 \int_{-\infty}^{+\infty} \left( \int_{{M}}\int_{{M}}\frac{\chi_{\{{u}>t\}}(x)\chi_{\{{u}\leq t\}}(y)}{|x-y|^{n+s}} \, {dx}{dy} \right) dt.
\end{align*}
Since
$$
\int_{{M}}\int_{{M}}\frac{\chi_{\{{u}>t\}}(x)\chi_{\{{u}\leq t\}}(y)}{|x-y|^{n+s}} \, {dx}{dy} = \int_{\{{u}>t\}}\int_{M \setminus \{{u}>t\}} \frac{{dy}{dx}}{|x-y|^{n+s}} = \Per_{{M},s}\left(\{u>t\}\right),
$$
we conclude the claim of the lemma.
\end{proof}

We can now give a

\begin{proof}[First proof of Theorem~\ref{thm:intro:Convex MS for functions} for~$p = 1$]
Without loss of generality we may assume ~$u$ to be non-negative. Indeed, the general case will then follow from this, by applying~\eqref{eq:Sec 4 main inequal} to~$|u|$ and noticing that~$\big| |u(x)| - |u(y)| \big| \le |u(x) - u(y)|$. We may also suppose that~$u$ has compact support---see the final argument in the proof of Theorem~\ref{thm:intro:Convex MS for functions} for~$p \ge 1$, presented later in this section, for details on how to remove this assumption.

From the expression
$$
u(x) = \int_0^{+\infty}\chi_{\{{u}>t\}}(x) \, dt,
$$
we use Minkowski's integral inequality to obtain that
\begin{align*}
\| u \|_{L^{\frac{n}{n - s}}(\delomega)} & = \left\|\int_0^{+\infty}\chi_{\{{u}>t\}} \, dt \, \right\|_{L^{\frac{n}{n-s}}({\delomega})} \\
& \le \int_0^{+\infty} \|\chi_{\{{u}>t\}}\|_{L^{\frac{n}{n-s}}({\delomega})} \, dt = \int_0^{+\infty} |\{ u > t\}|^{\frac{n - s}{n}} \, dt.
\end{align*}
We now apply the inequality of Proposition~\ref{pointwiseconvexMSprop} with~$E := \{ u > t \}$---observe that~$|\{ u > t\}| < +\infty$ as~$u$ has compact support. Integrating it over~$E$, we see that 
$$
  |\{{u}>t\}|^{\frac{n-s}{n}}\leq C\left(\Per_{{\delomega},s}(\{{u}>t\})+\int_{\{{u}>t\}} \, H_{\alpha}(x)^{\frac{s}{\alpha}} \, {dx}\right).
$$
By combining the last two estimates, we get that
\begin{equation}
 \label{eq:coarea form fbigger smaller t}
  \|{u}\|_{L^{\frac{n}{n - s}}({\delomega})} \le C\int_0^{+\infty}\left(\Per_{{\delomega},s}(\{{u}>t\})+\int_{\{{u}>t\}} H_{\alpha}(x)^{\frac{s}{\alpha}} \, {dx}\right)dt.
\end{equation}
Finally, by Fubini's theorem we have
\begin{align*}
 \int_0^{+\infty} \left( \int_{\{{u}>t\}}H_{\alpha}(x)^{\frac{s}{\alpha}} \, {dx} \right) dt
 =&
 \int_{{\delomega}\cap \{{u}>0\}} H_{\alpha}(x)^{\frac{s}{\alpha}} |{u}(x)| \,{dx}.
\end{align*}
Plugging this into~\eqref{eq:coarea form fbigger smaller t} and using Lemma~\ref{lem:frac coarea on manifold} we deduce
$$
  \|{u}\|_{L^{\frac{n}{n-s}}({\delomega})}\leq C\left(\frac{1}{2}
  \int_{{\delomega}}\int_{{\delomega}}\frac{|{u}(x)-{u}(y)|}{|x-y|^{n+s}} \, {dx}{dy}+
  \int_{{\delomega}}H_{\alpha}(x)^{\frac{s}{\alpha}}|{u}(x)| \,{dx}
  \right),
$$
This settles the theorem for $p=1$.
\end{proof}

We now present an adaptation of the slicing procedure of~\cite{Savin Valdinoci 1}. It will lead to the proof of Theorem~\ref{thm:intro:Convex MS for functions} in the general case $p \ge 1$. 

We first introduce some notation.
Let~${u}:{\delomega}\to \R$ be a bounded and non-negative measurable function with compact support. For~$i \in \Z$, we write
\begin{align*}
  A_i & :=\{{u}> 2^i\},\quad a_i:=|A_i|,
  \\
  D_i & := A_i\setminus A_{i+1}=\left\{ 2^i< {u}\leq 2^{i+1}\right\},\quad\text{ and }\quad d_i:=|D_i|.
\end{align*}
We have that the sets $D_i$ are pairwise disjoint,
$$
  \{ u = 0\} \cup \bigcup_{\substack{ j\in \mathbb{Z}\\j\leq i}}D_j= \partial \Omega \setminus A_{i+1},\quad
  \bigcup_{\substack{ j\in\mathbb{Z}\\ j\geq i}} D_j=A_i\,,\quad\text{ and }\quad 
  a_i=\sum_{\substack{j\in \mathbb{Z}\\ j\geq i}} d_j.
$$

We will need the following auxiliary lemma---see~\cite[Lemma 6.2]{HitchhikersG}
for its proof, which is very short and only uses H\"older's inequality.
Note that, as~$u$ is bounded, non-negative, and has compact support, our sequence~$\{ a_i \}$ satisfies the hypotheses of the lemma for some~$N\in\mathbb{Z}.$

\begin{lemma}
\label{lemma:ineq of sums ak and akplus1}
Let~$s\in (0,1),$~$p\geq 1$ such that~$s p<n$, and~$N\in\mathbb{Z}.$ Suppose~$\{ a_i \}_{i \in \Z}$ is a bounded, non-negative, and non-increasing sequence with 
$$
  a_i=0\quad\text{ for all~$i\geq N$}.
$$
Then
$$
  \sum_{i\in \mathbb{Z}} 2^{pi} a_i^{(n-s p)/n} \leq 2^{p^{\ast}_s}\sum_{\substack {i\in\mathbb{Z} \\ a_i\neq 0}}2^{pi}a_i^{-s p/n}a_{i+1}.
$$
\end{lemma}

Note that, from the hypotheses made on the sequence $\{a_{i}\}$ in the lemma, clearly both series are convergent.
The same happens for the series in the following inequality, which is taken from the proof of~\cite[Lemma 6.3]{HitchhikersG} and that we will use later:
\begin{equation}
 \label{eq:ineq for S}
\sum_{\substack{ i\in \Z \\ a_{i-1}\neq 0}} 2^{pi}a_{i-1}^{-s p/n}a_{i+1} \leq  
\frac{1}{2}\sum_{\substack{ i\in \Z \\ a_{i-1}\neq 0}} 2^{pi}a_{i-1}^{-s p/n}a_i .
\end{equation}
Its proof is simple:
$$
  \begin{aligned}
\sum_{\substack{ i\in \Z \\ a_{i-1}\neq 0}} 2^{pi}a_{i-1}^{-s p/n}a_{i+1} &=
\sum_{\substack{ i\in \Z \\ a_{i-1}\neq 0, a_{i+1}\neq 0}} 2^{pi}a_{i-1}^{-s p/n}a_{i+1} =\sum_{\substack{ i\in \Z \\ a_{i}\neq 0}} 2^{pi}a_{i-1}^{-s p/n}a_{i+1}  \\
&\leq  
\sum_{\substack{ i\in \Z \\ a_{i}\neq 0}} 2^{pi}a_{i}^{-s p/n}a_{i+1} =  
\frac{1}{2^{p}}\sum_{\substack{ j\in \Z \\ a_{j-1}\neq 0}} 2^{pj}a_{j-1}^{-s p/n}a_{j}
\leq \frac{1}{2}\sum_{\substack{ j\in \Z \\ a_{j-1}\neq 0}} 2^{pj}a_{j-1}^{-s p/n}a_{j}.
  \end{aligned}
  $$

The next lemma is the core of the proof and the analogue of~\cite[Lemma 6.3]{HitchhikersG}.

\begin{lemma}
\label{Analogue Lemma 6.3 in HG} Let~$s\in (0,1),$~$p \ge 1$ such that~$n > s p$, and~$\Omega\subset\R^{n+1}$ be an open convex set. Let~${u}\in L^{\infty}({\delomega})$ be a non-negative function with compact support. 
Then, 
\begin{align*}
 \frac{1}{2}\int_{{\delomega}}\int_{{\delomega}}\frac{|{u}(x)-{u}(y)|^p}{|x-y|^{n+s p}}\, {dx}{dy}
 +\int_{{\delomega}} H_{\alpha}(x)^{\frac{s p}{\alpha}} {u}(x)^p \, {dx} \ge c \sum_{\substack{i\in\Z \\ a_{i-1}\neq 0}}2^{pi}a_{i-1}^{-s p/n}a_i \,,
\end{align*}
for some constant~$c>0$ depending only on~$n$,~$\alpha$,~$s$, and~$p$.
\end{lemma}

\begin{proof}
Throughout the proof, we will use the notation~$D_{-\infty} := \{ u = 0 \}$. Moreover, for any~$k \in \Z$, we write~$j \le k$ to indicate that~$j$ is either an integer smaller than or equal to~$k$ or that~$j = -\infty$.

Let~$i\in \Z$ and~$x\in D_i$. For every~$j\leq i-2$ and~$y\in D_j$ we have that~${u}(x)-{u}(y) \ge 2^{i-1}$ and therefore
\begin{equation}
 \label{eq:fx minus fy by 2i}
  \begin{aligned}
  \sum_{j\leq i-2}\int_{D_j}\frac{|{u}(x)-{u}(y)|^p}{|x-y|^{n+s p}} \, {dy}
  & \geq 2^{p(i-1)}
  \sum_{j\leq i-2}\int_{D_j}\frac{dy}{|x-y|^{n+s p}} \\
  & =  2^{p(i-1)}\int_{\partial \Omega \setminus A_{i-1}}\frac{dy}{|x-y|^{n+ s p}}.
\end{aligned}
\end{equation}
Suppose now that~$A_{i-1}$ has positive measure. From Proposition~\ref{pointwiseconvexMSprop}, we have that
$$
  a_{i-1}^{-s p/n}\leq C\left(\int_{\delomega \setminus A_{i-1}}\frac{dy}{|x-y|^{n+ s p}} + H_{\alpha}(x)^{\frac{s p}{\alpha}} \right)
$$
for a.e.~$x\in A_{i-1}$. As a consequence, using~\eqref{eq:fx minus fy by 2i},
$$
  \sum_{j\leq i-2} \int_{D_j}\frac{|{u}(x)-{u}(y)|^p}{|x-y|^{n+s p}} \, {dy}
  \geq  2^{p(i-1)}\left(\frac{a_{i-1}^{-s p/n}}{C}- H_{\alpha}(x)^{\frac{sp}{\alpha}} \right)
$$
for a.e.~$x\in D_i\subset A_{i-1}$. Integrating over~$D_i$, this gives that
\begin{equation*}
 \label{eq:6.15 in HG}
  \begin{aligned}
& \sum_{j\leq i-2}\int_{D_i}\int_{D_j}\frac{|{u}(x)-{u}(y)|^p}{|x-y|^{n+s p}} \, {dy}{dx}+2^{p(i-1)}\int_{D_i} H_{\alpha}(x)^{\frac{s p}{\alpha}} \, {dx} \\
& \hspace{70pt} \ge \frac{2^{pi}a_{i-1}^{-s p/n}d_i}{2^p C} =  \frac{2^{pi}a_{i-1}^{-s p/n} (a_i-a_{i+1}) }{2^p C} .
 \end{aligned}
 \end{equation*}
We now take the sum over all~$i\in \Z$ such that~$a_{i-1}\neq 0$ and use \eqref{eq:ineq for S} to deduce
\begin{equation}
 \label{eq:2Q bigger}
  \begin{aligned}
  &\sum_{\substack{i\in \Z \\ a_{i-1}\neq 0}}
  \sum_{j\leq i-2}\int_{D_i}\int_{D_j}\frac{|{u}(x)-{u}(y)|^p}{|x-y|^{n+s p}} \, {dy}{dx}+\sum_{\substack{i\in \Z \\ a_{i-1}\neq 0}}2^{p(i-1)}\int_{D_i} H_{\alpha}(x)^{\frac{sp}{\alpha}} \, {dx}  \\
  & \hspace{70pt} \geq\frac{1}{2^p C} \,  \frac{1}{2} \sum_{\substack{i\in \Z \\ a_{i-1}\neq 0}}2^{pi}a_{i-1}^{-s  p/n} a_i\,.
   \end{aligned}
\end{equation}

Now, by symmetry
\begin{align} \label{eq:bigger 2Q}
  \frac{1}{2}\int_{{\delomega}}\int_{{\delomega}}\frac{|{u}(x)-{u}(y)|^p}{|x-y|^{n+ s p}} \, {dx}{dy}
  & \geq \sum_{i \in \Z} \sum_{j \le i - 1}
  \int_{D_i}\int_{D_j}\frac{|{u}(x)-{u}(y)|^p}{|x-y|^{n+ s p}} \, {dy}{dx}\nonumber
   \\
  & \geq \sum_{\substack{ i\in \Z \\ a_{i-1}\neq 0}}\sum_{j\leq i-2}\int_{D_i}\int_{D_j}\frac{|{u}(x)-{u}(y)|^p}{|x-y|^{n+s p}} \, {dy}{dx}.
\end{align}
Since~${u}(x)^p>2^{pi}>2^{p(i-1)}$ for $x\in D_i$, we have
\begin{align*}
 \int_{D_i}H_{\alpha}(x)^{\frac{s p}{\alpha}} {u}(x)^p \, {dx}
 \geq
  2^{p (i-1)}\int_{D_i} H_{\alpha}(x)^{\frac{s p}{\alpha}} \, {dx}.
\end{align*}
Thus,
\begin{align*}
  \int_{{\delomega}} H_{\alpha}(x)^{\frac{s p}{\alpha}}{u}(x)^p \, {dx} 
  \geq
  \sum_{\substack{i\in \Z \\ a_{i-1}\neq 0}}\int_{D_i} H_{\alpha}(x)^{\frac{s p}{\alpha}}{u}(x)^p \, {dx} 
  \ge 
  \sum_{\substack{i\in \Z \\ a_{i-1}\neq 0}}2^{p(i-1)}\int_{D_i} H_{\alpha}(x)^{\frac{s p}{\alpha}} \, {dx}.
\end{align*}
It now follows from this,~\eqref{eq:bigger 2Q}, and~\eqref{eq:2Q bigger} that
$$
  \frac{1}{2}\int_{{\delomega}}\int_{{\delomega}}\frac{|{u}(x)-{u}(y)|^p}{|x-y|^{n+ s p}} \, {dx}{dy}+\int_{{\delomega}} H_{\alpha}(x)^{\frac{s p}{\alpha}} {u}(x)^p \, {dx} \geq \frac{1}{2^{p+1} C}\sum_{\substack{ i\in \Z \\ a_{i-1}\neq 0}}2^{pi}a_{i-1}^{- s p/n}a_i\,,
$$
which concludes the proof of the lemma.
\end{proof}

\begin{proof}[Proof of Theorem~\ref{thm:intro:Convex MS for functions}.]
As for the proof in the case~$p = 1$ presented previously, we may assume~$u$ to be non-negative. Using truncations, we can also take~${u}$ to be bounded. In addition, we suppose for the moment that~$u$ has compact support. We will show at the end of the proof that this hypothesis can be removed.

Under these assumptions, we have
$$
  \|{u}\|_{L^{p^{\ast}_s}({\delomega})}^{p^{\ast}_s}= \sum_{i\in \Z}\int_{D_i}{u}(x)^{p^{\ast}_s} \, {dx}
  \leq \sum_{i\in \Z}\int_{D_i}\left(2^{i+1}\right)^{p^{\ast}_s} {dx} \leq \sum_{i\in \Z}2^{p^{\ast}_s(i+1)} a_i.
$$
From this and the elementary inequality~$\left(\sum_{i} m_i\right)^{\lambda}\leq \sum_{i} m_i^{\lambda}$ for every sequence~$m_i\geq 0$ and~$\lambda\in [0,1],$ taking here~$ \lambda := p/p^{\ast}_s=(n- s p)/n \in (0,1)$, one concludes that
$$
  \|{u}\|_{L^{p^{\ast}_s}({\delomega})}^p\leq 2^p\left(\sum_{i\in \Z}2^{p^{\ast}_s i} a_i\right)^{p/p^{\ast}_s}\leq 2^p\sum_{i\in \Z}2^{pi}a_i^{(n- s p)/n}.
$$
Using now Lemmas~\ref{lemma:ineq of sums ak and akplus1} and~\ref{Analogue Lemma 6.3 in HG} we get
\begin{align*}
  \|{u}\|_{L^{p^{\ast}_s}({\delomega})}^p
  & \leq
  2^{p+p^{\ast}_s}\sum_{\substack{i\in \Z \\ a_i\neq 0}}2^{pi}a_i^{- s p/n} a_{i+1}
  =2^{p^{\ast}_s}\sum_{\substack{i\in \Z \\ a_{i-1}\neq 0}} 2^{pi} a_{i-1}^{- s p/n} a_i 
  \\
  & \leq C \left(\frac{1}{2}\int_{{\delomega}}\int_{{\delomega}}\frac{|{u}(x)-{u}(y)|^p}{|x-y|^{n+ s p}} \, {dx}{dy}
 +\int_{{\delomega}} H_{\alpha}(x)^{\frac{s p}{\alpha}} {u}(x)^p \, {dx}\right),
\end{align*}
which proves the theorem under the assumption that~$u$ has compact support.

We now show that the compactness of~$\supp(u)$ is not needed. Let~$R \ge 1$ and consider a cutoff function~$\eta \in C^\infty_c(\R^{n + 1})$ satisfying~$0 \le \eta \le 1$ in~$\R^{n + 1}$,~$\eta = 1$ in~$B_R$,~$\supp(\eta) \subset B_{2 R}$, and~$|\nabla \eta| \le 2/R.$ Given~$u \in W^{s, p}(\delomega)$, we define~$v := \eta u$. By the inequality that we have just proved and since~$v$ has compact support, we have that
\begin{equation} \label{sobandcutoff}
\begin{aligned}
  \|u\|_{L^{p^{\ast}_s}(\delomega \cap B_R)}^p & \le \| v \|_{L^{p^{\ast}_s}(\delomega)}^p 
  \\
  & 
  \le C\left( \frac{1}{2}\int_{\delomega}\int_{\delomega} \frac{|v(x)-v(y)|^p}{|x-y|^{n+ s p}} \, dx dy +\int_{\delomega}H_{\alpha}(x)^{\frac{s p}{\alpha}} |v(x)|^p \, dx \right)
  \\
  & \le C \, \bigg(\frac{1}{2} \int_{\delomega}\int_{\delomega} \frac{|u(x)-u(y)|^p}{|x-y|^{n+ s p}} \, dx dy +\int_{\delomega}H_{\alpha}(x)^{\frac{s p}{\alpha}} |u(x)|^p \, dx 
  \\
  & \quad\, + \int_{\delomega} |u(x)|^p \left( \int_{\delomega} \frac{|\eta(x) - \eta(y)|^p}{|x - y|^{n + s p}} \, dy \right) dx \bigg),
\end{aligned}
\end{equation}
where for the last inequality we used that
$$
|v(x)-v(y)|^p \le 2^{p - 1} \Big( |u(x) - u(y)|^p + |u(x)|^p |\eta(x) - \eta(y)|^p \Big) \quad \mbox{for a.e.~} x, y \in \delomega.
$$
To control the last term in~\eqref{sobandcutoff} we adapt some techniques from~\cite[Subsection~3.2]{Cabre-Cozzi}. First, using the Lipschitz property of~$\eta$ we have
\begin{equation} \label{sobandcutoff2}
   \int_{\delomega} \frac{|\eta(x) - \eta(y)|^p}{|x - y|^{n + s p}} \, dy \le \frac{2^p}{R^p} \int_{\delomega \cap B_R(x)} \frac{dy}{|x - y|^{n - (1 - s) p}} + \int_{\delomega \setminus B_R(x)} \frac{dy}{|x - y|^{n + s p}}.
\end{equation}
To estimate the second term on the right, we argue similarly to~\cite[Lemma~3.3]{Cabre-Cozzi}. Taking advantage of the perimeter estimate~\eqref{eq:perestcor}, we deduce
\begin{align*}
\int_{\delomega \setminus B_R(x)} \frac{dy}{|x - y|^{n + s p}} & = \sum_{j = 1}^{+\infty} \int_{\delomega \cap \left( B_{2^j \! R}(x) \setminus B_{2^{j - 1} \! R}(x) \right)} \frac{dy}{|x - y|^{n + s p}} \le \sum_{j = 1}^{+\infty} \frac{|\delomega \cap B_{2^j \! R}(x)|}{\left( 2^{j - 1} \! R \right)^{n + s p}} \\
& \le \frac{2^{n + s p} |\mathbb{S}^n|}{R^{s p}} \sum_{j = 1}^{+\infty} 2^{- sp j} \le \frac{C}{R^{s p}}.
\end{align*}
As the first term on the right-hand side of~\eqref{sobandcutoff2} can be dealt with using~\cite[Lemma~3.4]{Cabre-Cozzi}---observe that hypothesis~(3.3) of~\cite{Cabre-Cozzi} is fulfilled thanks to our~\eqref{eq:perestcor}---we infer that
$$
\int_{\delomega} \frac{|\eta(x) - \eta(y)|^p}{|x - y|^{n + s p}} \, dy \le \frac{C}{R^{s p}}.
$$
By plugging this into~\eqref{sobandcutoff} and letting~$R \rightarrow +\infty$, we conclude that~$u$ satisfies~\eqref{eq:intro:sAMSonconvex}. The proof is thus complete.
\end{proof}

\section{Application to the fractional mean curvature flow}
\label{section:fractional mean curv flow}

\noindent
In this section we study the evolution of convex sets under fractional mean curvature flow. Using the pointwise inequality~\eqref{eq:old pointwise aleks Fench} in conjunction with the classical Michael-Simon inequality, we provide an upper bound for the maximal time of existence for the smooth fractional mean curvature flow of convex hypersurfaces. Namely, we prove Theorem~\ref{thm:extinction time}. As in the classical local case, the argument is simple, once the appropriate Michael-Simon type inequality is known.

We denote by~$\Omega_0 \subset \R^{n + 1}$ a bounded open convex set with~$C^2$ boundary, and by~$\Omega_t$ its evolution by fractional~$\alpha$-mean curvature flow. That is, the inner normal velocity is, at every point, the fractional~$\alpha$-mean curvature. The unit outer normal to~$\Omega_t$ is denoted by~$\nu_t$ and we take the mean curvature~$H$ of~$\Omega$ (i.e., the sum of its principal curvatures) with the sign convention to be non-negative for convex sets.

As in~\eqref{max-time}, we consider
$$
T^\ast := \sup \left\{ t > 0 : \Omega_\tau \mbox{ is non-empty and has~$C^2$ boundary for all~} \tau\in [0,t) \right\}.
$$
In view of the results of~\cite{JulinLaManna},~$\Omega_\tau$ has boundary of class~$C^2$---actually,~$C^\infty$---for every small~$\tau$. Hence,~$T^\ast > 0$. On the other hand, through comparison with shrinking balls in~\cite{Saez Valdinoci} it is proved that~$T^\ast \le C \diam(\Omega_0)^{1 + \alpha}$ for some constant~$C$ depending only on~$n$ and~$\alpha$. These two results hold regardless of the convexity of~$\Omega_0$. Here, we show that, when~$\Omega_0$ is convex, the bound on~$T^\ast$ can be improved to~\eqref{eq:boundforTast}.

First, we recall a general first variation formula. In our situation, we will apply it with~$\varphi_t = - H_\alpha[\Omega_t]$. Note that, throughout this section, we emphasize the dependence of the classical an fractional mean curvatures on the set~$\Omega$ by writing~$H(x) = H[\Omega](x)$ and~$H_\alpha(x) = H_\alpha[\Omega](x)$ for~$x \in \partial \Omega$.

\begin{lemma}[See, e.g.,~{\cite[Remark~4.2]{Ecker}} or~{\cite[Proposition~4]{Guan-Li}}]
\label{Lemma:first variation mean curv flow}
Let~$\Omega_t\subset\R^{n+1}$ be a one-parameter family of open sets with~$C^2$ boundary and with~$|\partial \Omega_t|<+\infty$ for all $t \in (-a, a)$ and some~$a > 0$. Assume that, corresponding to each point~$p_0\in \partial \Omega_0$, there is a differentiable curve~$t\mapsto p(t)$ with~$p(0)=p_0$,~$p(t) \in \partial \Omega_t$ for all $t \in (-a, a)$, and satisfying
$$
  \frac{d}{dt} \, p(t)= \varphi_t(p(t)) \, \nu_t(p(t)) \quad \text{for all } t \in (-a, a),
$$
for some continuous function~$\varphi_t: \delomega_t\to \R$.

Then,
$$
  \frac{d}{dt} |\delomega_t|=\int_{\delomega_t} \varphi_t \, H[\Omega_t] \, {dx}.
$$ 
\end{lemma}

We can now give the

\begin{proof}[Proof of Theorem~\ref{thm:extinction time}.]
Recall that~$\Omega_t$ remains convex, thanks to~\cite{Chambolle Novaga Ruffini}. Using Lemma~\ref{Lemma:first variation mean curv flow} we see that
\begin{equation}
 \label{eq:decreasing sigma area}
  \frac{d}{dt} |{\delomega}_t|=-\int_{{\delomega}_t} {\HalphaOmegat} H[\Omega_t] \, {dx}.
\end{equation}
By inequality~\eqref{eq:old pointwise aleks Fench} proved in Section~\ref{section:frac Aleks Fench}, we know that
$$
  |{\delomega}_t|^{-\frac{\alpha}{n}}\leq C_1 {\HalphaOmegat}(x)\quad\text{ for all }x\in {\delomega}_t,
$$
for some constant~$C_1>0$ depending only on~$n$ and~$\alpha$. Multiplying this inequality by~$H[\Omega_t](x)$ and integrating in~$x \in \delomega_t$, we get
\begin{equation}
 \label{eq:using Prop 9 for Sigmat}
  |{\delomega}_t|^{-\frac{\alpha}{n}}\int_{{\delomega}_t}H[\Omega_t] \, {dx}
  \leq C_1 \int_{{\delomega}_t}{\HalphaOmegat} H[\Omega_t] \, {dx}.
\end{equation}
We now use the the classical Michael-Simon inequality (Theorem~\ref{thm:Michael-Simon Allard}) with~$u \equiv 1 = p$ if~$n\geq 2$, or the Gauss-Bonnet formula for curves: $2\pi=\int_{\partial\Omega_t}H[\Omega_t](x) \, dx$ if~$n=1.$ Either way, we have that
\begin{equation}
 \label{eq:using classic MS for Sigmat}
  |{\delomega}_t|^{\frac{n-1}{n}}\leq C_2
  \int_{{\delomega}_t}H[\Omega_t] \, {dx}
\end{equation}
for some constant~$C_2>0$ depending only on~$n$.

Finally, using~\eqref{eq:using classic MS for Sigmat},~\eqref{eq:using Prop 9 for Sigmat}, and~\eqref{eq:decreasing sigma area}, we deduce that
\begin{align*}
  |{\delomega}_t|^{\frac{n-(1+\alpha)}{n}}
  & = |{\delomega}_t|^{\frac{n-1}{n}}|{\delomega}_t|^{-\frac{\alpha}{n}}
  \leq C_2\,|{\delomega}_t|^{-\frac{\alpha}{n}}\int_{{\delomega}_t}H[\Omega_t] \, {dx}
  \\
  & \le C_1C_2 \int_{{\delomega}_t}{\HalphaOmegat}H[\Omega_t] \, {dx}
  =-C_1 C_2 \, \frac{d}{dt} |{\delomega}_t|.
\end{align*}
That is,~$\frac{d}{dt} |\delomega_t|^{\frac{1 + \alpha}{n}} \le - \delta$, for some constant~$\delta > 0$ depending only on~$n$ and~$\alpha$. By integrating this relation, we obtain that~$|{\delomega}_t|^{\frac{1+\alpha}{n}}\leq |{\delomega}_0|^{\frac{1+\alpha}{n}}- \delta t$. This shows that the maximal time of existence must satisfy~$T^\ast \le \delta^{-1} |{\delomega}_0|^{\frac{1+\alpha}{n}}$, as claimed by the theorem.
\end{proof}

\appendix
\section{Proof of the Rosenthal-Sz\'asz type inequality}

\label{appendix:rosenthal-szasz}

\noindent
In this section, we denote by~$B_1^n$ the open unit ball of~$\R^n$ centered at the origin, that is~$B_1^n := \{ x\in\R^n:\, |x|<1\}$. Here we give a proof of the first inequality in Proposition~\ref{RSineprop} (the isodiametric inequality for perimeter), which states that
\begin{equation}
 \label{appendix:RSzasz type ineq}
  |\partial\Omega|\leq |\mathbb{S}^n| \, \frac{\diam(\Omega)^n}{2^n}\quad \mbox{for every bounded convex set } \Omega\subset\R^{n+1}.
\end{equation}

Observe that the inequality is optimal, i.e.,~there is equality for balls. This inequality was first proved by Rosenthal and Sz\'asz~\cite{Rosenthal Szasz} in the plane. The version in higher dimensions can be found in Section~44 of~\cite{BF87} as inequality~(6). The proof however is scattered over several sections of~\cite{BF87}, of which many steps are in greater generality than what is actually needed to prove~\eqref{appendix:RSzasz type ineq}, making the proof unnecessarily long and  complicated if one is only interested in the Rosenthal-Sz\'asz inequality. We have not found a better reference and, thus, we present here a quick proof. It is based on two better-known results: Cauchy's surface area formula (Proposition \ref{proposition:SK by int sgima u} below) and the isodiametric inequality for volume. This last result---see, e.g., Theorem~1 in Section~2.2 of~\cite{Evans-Gariepy}---states that
\begin{equation}
\label{eq:isodiam ineq volume}
  |E| \leq |B_1^n|\, \frac{\diam(\Omega)^{n}}{2^{n}},
\end{equation}
where~$|\cdot|$ indicates the~$n$-dimensional Lebesgue measure and~$E\subset\R^n$ is any measurable set---here convexity is not needed. In~\cite{Evans-Gariepy} it is proved using Steiner symmetrizations. As for~\eqref{appendix:RSzasz type ineq}, in~\eqref{eq:isodiam ineq volume} equality is achieved for balls. Observe that the isodiametric inequality for perimeter does not hold in general if the convexity assumption is relaxed. Consider for example a domain with oscillating boundary---giving an arbitrary large perimeter---contained in a ball of a given diameter.

If one does not need the best constant in the Rosenthal-Sz\'asz inequality~\eqref{appendix:RSzasz type ineq}---as it is our case---, a weaker inequality follows more easily from the inclusion~$\Omega \subset \overline{B}_{\diam(\Omega)}(x)$, where~$x$ is any point in~$\overline{\Omega}$, and the monotonicity of the perimeter with respect to the inclusion of convex sets. This monotonicity property follows, for instance, from Cauchy's surface area formula, stated later in Proposition~\ref{proposition:SK by int sgima u}. Given our statement of this result, one also needs to approximate the convex set by polytopes, as we do in the proof of Proposition~\ref{RSineprop} below.

For the proof of \eqref{appendix:RSzasz type ineq} we need to introduce the notion of polytopes.
A bounded open set~$K\subset\R^{n+1}$ is called a polytope if its boundary~$\partial K$ is the finite union of sets~$P_i$, for~$i=1,\ldots,m_K$, with each~$P_i$ being contained in an~$n$-dimensional affine hyperplane. The~$P_i$'s are the~$n$-dimensional faces of~$K$. In this section~$K\subset\R^{n+1}$ always denotes a convex polytope. Now, given a unit vector~$\sigma\in \mathbb{S}^n$, let~$K_\sigma$ be the projection of~$K$ onto the hyperplane orthogonal to~$\sigma$. Obviously, we have
\begin{equation}
 \label{eq:sigma delta K by sum Pi}
  |\partial K|=\sum_{i=1}^{m_K}|P_i|.
\end{equation}
Denote by~$\xi_i$ a unit normal vector on~$P_i$. Note that projecting the~$n$-dimensional faces~$P_i$ onto the hyperplane orthogonal to~$\sigma$ and then taking the union over~$i$ also coincides with~$K_{\sigma}$. At the same time, by the convexity of~$K$, the preimage of a.e.~$x\in K_{\sigma}$ under this projection consists of exactly two points lying on two different faces. Thus, we obtain the identity
\begin{equation}
 \label{eq:sigmaK u by faces}
  2|K_{\sigma}|=\sum_{i=1}^{m_K}|P_i|\,|\langle\xi_i,\sigma\rangle|.
\end{equation}

We will use the following lemma. 

\begin{lemma}
\label{lemma:int uv over Sn}
Let~$\tau\in \mathbb{S}^n$ be a unit vector in~$\R^{n+1}.$
Then
$$
  \int_{\mathbb{S}^n}|\langle \sigma,\tau\rangle| \, d\sigma=2|B_1^n|.
$$
\end{lemma}

\begin{proof}
After a rotation, we can assume~$\tau=e_{n+1}=(0,\ldots,0,1).$ Using the parametrization~$\varphi:B_1^n\to \mathbb{S}^n,$ given by~$\varphi(x_1,\ldots,x_n)=\left(x_1,\ldots,x_n, \sqrt{1-|x|^2}\right)$, we see that
$$
  \int_{\mathbb{S}^n}|\langle \sigma,\tau\rangle| \, d\sigma
  = 2
  \int_{\mathbb{S}^n \cap \{ \sigma_{n + 1} > 0 \}} \sigma_{n+1} \, d\sigma=2\int_{B_1^n} \sqrt{1-|x|^2}\,\sqrt{1 + |\nabla \varphi^{n + 1}(x)|^2} \, dx.
$$
The claim follows as~$\sqrt{1 + |\nabla \varphi^{n + 1}(x)|^2} = 1/ \sqrt{1 - |x|^2}$ for every~$x \in B_1^n$.
\end{proof}

As a result of the previous considerations, we have the following identity for the perimeter of~$K$, which is known as Cauchy's formula (see for instance \cite[page 89]{Eggeleston Convexity}).

\begin{proposition}[Cauchy's surface area formula]
\label{proposition:SK by int sgima u}
Let~$K\subset\R^{n+1}$ be a convex polytope. Then, it holds
$$
  |\partial K|=\frac{1}{|B_1^n|}\int_{\mathbb{S}^n}|K_{\sigma}| \, d\sigma.
$$
\end{proposition}

\begin{proof}
Integrate~\eqref{eq:sigmaK u by faces} with respect to~$\sigma$ over~$\mathbb{S}^n,$ then apply Lemma~\ref{lemma:int uv over Sn} (with~$\tau=\xi_i$), and finally use~\eqref{eq:sigma delta K by sum Pi}.
\end{proof}

We can finally give the

\begin{proof}[Proof of Proposition~\ref{RSineprop}.]
To prove~\eqref{appendix:RSzasz type ineq} we can assume by approximation that~$\Omega$ is a convex polytope~$K$ (see for instance \cite[Section 22]{Lay Convex sets} on approximations by polytopes). For any direction~$\sigma\in \mathbb{S}^n$ it follows from the isodiametric inequality for volume \eqref{eq:isodiam ineq volume} that
$$
  |K_{\sigma}| \leq |B_1^n| \, \frac{\diam (K_{\sigma})^n}{2^n}.
$$
Using now that~$\diam(K_{\sigma}) \le \diam(K)$ and Proposition~\ref{proposition:SK by int sgima u}, we get
$$
  |\partial K|\leq \left( \int_{\mathbb{S}^n}d\sigma \right) \frac{\diam(K)^n}{2^n}
  =|\mathbb{S}^n| \, \frac{\diam(K)^n}{2^n},
$$
and the proposition is proved.
\end{proof}


\begin{thebibliography}{99}


\bibitem{Aleksandrov 1} A.D. Aleksandrov, \emph{Zur Theorie der gemischten Volumina von konvexen K\"orpern, II, Neue
Ungleichungen zwischen den gemischten Volumina und ihre Anwendungen}, Mat. Sb. (N.S.) \textbf{2}
(1937), no.~6, 1205--1238.

\bibitem{Aleksandrov 2}
A.D. Aleksandrov, \emph{Zur Theorie der gemischten Volumina von konvexen K\"orpern, III, Die
Erweiterung zweeier Lehrsatze Minkowskis \"uber die konvexen Polyeder auf beliebige konvexe
Fl\"achen}, Mat. Sb. (N.S.) \textbf{3} (1938), no. 1, 27--46.

\bibitem{Allard} W.K. Allard, \emph{On the first variation of a varifold}, Ann. of Math. \textbf{95} (1972), no. 3, 417--491.

\bibitem{A55}
N. Aronszajn,
\emph{Boundary values of functions with finite Dirichlet integral},
Techn. Report \textbf{14}, Univ. of Kansas (1955), 77--94.

\bibitem{BDM}
E. Bombieri, E. De Giorgi, M. Miranda,
\emph{Una maggiorazione a priori relativa alle ipersuperfici minimali non parametriche},
Arch. Rational Mech. Anal. \textbf{32} (1969), 255--267.

\bibitem{BG73}
E. Bombieri, E. Giusti,
\emph{Local estimates for the gradient of non-parametric surfaces of prescribed mean curvature},
Comm. Pure Appl. Math. \textbf{26} (1973), 381--394.

\bibitem{BF87}
T. Bonnesen, W. Fenchel,
\emph{Theory of convex bodies},
BCS Associates, Moscow, ID, 1987.

\bibitem{Brendle} S. Brendle, \emph{The isoperimetric inequality for a minimal hypersurface in Euclidean space}, ArXiv preprint, arXiv:1907.09446, 2019.

\bibitem{brezis}
H. Brezis,
\emph{A quick proof of the fractional Sobolev inequality}, 2001,
unpublished, communicated by the author.

\bibitem{Cabre-DCDS}
X. Cabr\'e,
\emph{Elliptic PDE's in probability and geometry: symmetry and regularity of solution},
Discrete Contin. Dyn. Syst. \textbf{20} (2008), no. 3, 425--457.

\bibitem{Cabre-Cozzi}X. Cabr\'e, M. Cozzi, \emph{A gradient estimate for nonlocal minimal graphs}, Duke Math. J. \textbf{168} (2019), no. 5, 775--848.

\bibitem{Cabre-Miraglio} X. Cabr\'e, P. Miraglio, \emph{Universal Hardy-Sobolev inequalities on hypersurfaces of Euclidean space}, ArXiv preprint, arXiv:1912.09282, 2019.

\bibitem{Caffarelli Roquejoffre Savin} L. Caffarelli, J.-M. Roquejoffre, O. Savin, \emph{Nonlocal minimal surfaces}, Comm. Pure Appl. Math. \textbf{63} (2010), no. 9, 1111--1144.

\bibitem{Caffarelli Souganidis} L. Caffarelli, P. Souganidis, \emph{Convergence of nonlocal threshold dynamics approximations to front propagation}, Arch. Ration. Mech. Anal. \textbf{195} (2010), no. 1, 1--23.

\bibitem{CDNV19}
A. Cesaroni, S. Dipierro, M. Novaga, E. Valdinoci,
\emph{Fattening and nonfattening phenomena for planar nonlocal curvature flows},
Math. Ann. \textbf{375} (2019), no. 1-2, 687--736.

\bibitem{Chambolle Novaga Ruffini}
A. Chambolle, M. Novaga, B. Ruffini, \emph{Some results on anisotropic fractional mean curvature flows}, Interfaces Free Bound. \textbf{19} (2017), no. 3, 393--415.

\bibitem{Alice Chang}S.-Y. A. Chang and Y. Wang, \emph{On Aleksandrov-Fenchel Inequalities for k-Convex Domains}, Milan J. Math. \textbf{79} (2011), 13--38.

\bibitem{CintiSinestrariValdinoci}
E. Cinti, C. Sinestrari, E. Valdinoci,
\emph{Neckpinch singularities in fractional mean curvature flows},
Proc. Amer. Math. Soc. \textbf{146} (2018), no. 6, 2637--2646.

\bibitem{Dierkes}
U. Dierkes, S. Hildebrandt, A. Tromba, \emph{Global analysis of minimal surfaces}, Revised and enlarged second edition, Grundlehren der Mathematischen Wissenschaften [Fundamental Principles of Mathematical Sciences], 341, Springer, Heidelberg, 2010.

\bibitem{HitchhikersG} E. Di Nezza, G. Palatucci, E. Valdinoci, \emph{Hitchhiker's guide to the fractional Sobolev spaces}, Bull. Sci. Math. \textbf{136} (2012), no. 5, 521--573.

\bibitem{DILTV18}
B. Dyda, L. Ihnatsyeva, J. Lehrb{\"a}ck, H. Tuominen, A. V. V{\"a}h{\"a}kangas,
\emph{Muckenhoupt~$A_p$-properties of distance functions and applications to Hardy-Sobolev type inequalities},
Potential Anal. \textbf{50} (2019), no. 1, 83--105. 

\bibitem{Ecker}
K. Ecker,
\emph{Regularity theory for mean curvature flow},
Progress in Nonlinear Differential Equations and their Applications, 57. Birkh\"{a}user Boston, Inc., Boston, MA, 2004.

\bibitem{Eggeleston Convexity}
H. G. Eggleston, \emph{Convexity}, Cambridge Tracts in Mathematics and Mathematical Physics, No. 47 Cambridge University Press, New York, 1958. 

\bibitem{Evans CIME} L.C. Evans, \emph{Regularity for fully nonlinear elliptic equations and motion by mean curvature}, Lecture Notes in Math., 1660, Fond. CIME Subser., Springer, Berlin, 1997, 98--133.

\bibitem{Evans-Gariepy}
L.C. Evans, R. Gariepy,
\emph{Measure theory and fine properties of functions},
Studies in Advanced Mathematics, CRC Press, Boca Raton, FL, 1992.

\bibitem{Evans-Spruck}
L.C. Evans, J. Spruck, \emph{Motion of level sets by mean curvature III}, J. Geom. Anal. \textbf{2} (1992), no. 2, 121--150.

\bibitem{Folland PDE} G. Folland, 
\emph{Introduction to partial differential equations},
Second edition, Princeton University Press, Princeton, NJ, 1995.

\bibitem{G57}
E. Gagliardo,
\emph{Caratterizzazioni delle tracce sulla frontiera relative ad alcune classi di funzioni in~$n$ variabili},
Rend. Sem. Mat. Univ. Padova \textbf{27} (1957), 284--305. 

\bibitem{Guan-Li}
P. Guan, J. Li,
\emph{The quermassintegral inequalities for k-convex starshaped domains},
Adv. Math. \textbf{221} (2009), no. 5, 1725--1732. 

\bibitem{huisken}
G. Huisken,
\emph{Flow by mean curvature of convex surfaces into spheres},
J. Differential Geom. \textbf{20} (1984), no. 1, 237--266.

\bibitem{Imbert level set}
C. Imbert, \emph{Level set approach for fractional mean curvature flows}, Interfaces Free Bound. \textbf{11} (2009), no. 1, 153--176.
 
\bibitem{JulinLaManna}
V. Julin, D. La Manna,
\emph{Short time existence of the classical solution to the fractional mean curvature flow},
ArXiv preprint, arXiv:1906.10990, 2019.

\bibitem{Lay Convex sets}
S. Lay, \emph{Convex sets and their applications}, Pure and Applied Mathematics, A Wiley-Interscience Publication, John Wiley and Sons, Inc., New York, 1982.
 
\bibitem{M12}
F. Maggi,
\emph{Sets of finite perimeter and geometric variational problems. An introduction to geometric measure theory},
Cambridge Studies in Advanced Mathematics, Vol. 135, Cambridge University Press, Cambridge, 2012.

\bibitem{Michal Simon}
J.H. Michael, L.M. Simon, \emph{Sobolev and mean-value inequalities on generalized subamnifolds of~$\R^n$}, Comm. Pure Appl. Math. \textbf{26} (1973), 361--379.

\bibitem{Rosenthal Szasz} A. Rosenthal, O. Sz{\'a}sz, \emph{Eine Extremaleigenschaft der Kurven konstanter Breite}, Jahresber. Deutsch. Math.-Verein. \textbf{25} (1917), 278--282.


\bibitem{Saez Valdinoci} M. S\'aez, E. Valdinoci, \emph{On the evolution by Fractional Mean Curvature},  Comm. Anal. Geom. \textbf{27} (2019), no. 1, 211--249.

\bibitem{Savin Valdinoci 1}
O. Savin, E. Valdinoci, \emph{Density estimates for a nonlocal
variational model via the Sobolev inequality},
 SIAM J. Math. Anal. \textbf{43} (2011), no. 6, 2675--2687.

\bibitem{SV14}
O. Savin, E. Valdinoci,
\emph{Density estimates for a variational model driven by the {G}agliardo norm},
J. Math. Pures Appl. (9) \textbf{101} (2014), no. 1, 1--26.

\bibitem{SB56}
L. N. Slobodecki{\u{\i}}, V. M. Babi{\u{c}},
\emph{On boundedness of the Dirichlet integrals},
Dokl. Akad. Nauk SSSR (N.S.) \textbf{106} (1956), 604--606.

\bibitem{Visintin}
A. Visintin,
\emph{Generalized coarea formula and fractal sets},
Japan J. Indust. Appl. Math. \textbf{8} (1991), no. 2, 175--201.

\end{thebibliography}
\end{document}